\documentclass{amsart}
\usepackage{amsmath}
\usepackage{amsthm}
\usepackage{amssymb}
\usepackage{latexsym}               
\usepackage{amsgen}
\usepackage[mathscr]{eucal}
\numberwithin{equation}{section}
\usepackage[dvipdfm, usenames]{color} 

\newcommand{\al}{\alpha}

\newcommand{\e}{\epsilon}

\newcommand{\ti}{\tilde}

\makeatother

\newtheorem{lem}{Lemma}[section]
\newtheorem{thm}[lem]{Theorem}
\newtheorem{pro}[lem]{Proposition}
\newtheorem{cor}[lem]{Corollary}

\newtheorem{remark}{Remark}

\begin{document}

\title
{Inviscid Limit Problem for radially symmetric stationary states of viscous and heat-conductive ideal gas}
\author{Itsuko Hashimoto}
\email{hashimoto@math.sci.osaka-u.ac.jp}

\noindent
\address{Department of Mathematics, Graduate School of Science,
The University of Osaka, 1-1 Machikaneyama-cho, Toyonaka,
Osaka 560-0043, Japan}

\noindent
\address{Osaka Central Advanced Mathematical Institute (OCAMI),
Osaka Metropolitan University, 3-3-138 Sugimoto, Sumiyoshi-ku,
Osaka 558-8585, Japan}


\date{}
\maketitle

\def\cal#1{{\fam2#1}}



\begin{abstract} 

We study the zero-viscosity limit of radially symmetric stationary
solutions for the inflow and outflow problems for a viscous and heat-conductive polytropic
ideal gas in the exterior of a ball in $\mathbb{R}^n$ $(n\geq 3)$,
under prescribed boundary and far-field conditions.
The far-field velocity is assumed to be zero, and the solutions are
considered in a sufficiently small neighborhood of the far-field state. 
In the zero-viscosity limit, the viscous terms disappear from the
momentum equation, while the heat-conduction term remains in the
energy equation.
For the outflow problem, we prove that the viscous stationary solution
converges uniformly to the corresponding stationary solution of the
heat-conductive inviscid system, with convergence rate $O(\mu)$.
For the inflow problem, a boundary layer appears near the inner
boundary.
We show that the viscous stationary solution is uniformly approximated
by the superposition of the limiting stationary solution and a
boundary-layer profile. 
Here, the boundary layer appears only for the density and 
velocity, whereas no boundary-layer appears for the temperature. 
Moreover, for any $\alpha\in(0,1)$, the approximation error is of order
$O(\mu^\alpha)$.
%
%
\end{abstract}

\medskip


\section{Introduction and Main theorem}
\noindent
We study the inviscid limit of radially symmetric stationary flows
of a viscous and heat-conductive ideal gas in an exterior domain.
In the present model, the temperature is treated as an independent thermodynamic
variable and is coupled with the density and velocity through the equation of state
and the energy balance.
More precisely, we consider the system for a Newtonian polytropic ideal-gas  
in the exterior of a ball in $\mathbb{R}^n$ $(n\ge3)$:
\begin{eqnarray}
 \label{cns} 
 \left\{
\begin{array}{l}
\rho_t+\nabla\cdot (\rho U)=0, \\[10pt]
(\rho U)_t+\nabla\cdot(\rho U\otimes U)+\nabla p =\nabla \cdot D,\quad
t>0,\ x \in \Omega,\\[10pt]
\left\{\rho \left(e + \frac{1}{2}|U|^2\right)\right\}_t +
\nabla\cdot \left\{\rho U \left(e + \frac{1}{2}|U|^2\right)\right\}+\nabla\cdot(pU)= \nabla\cdot (DU -q),\\[10pt]
\end{array}
 \right.\, 
\end{eqnarray} 
where
$
\Omega=\{x\in\mathbb{R}^{n}\,(n\ge 3);\, |x|>r_0\}
$ 
for a fixed $r_0>0$,   
the unknown functions $\rho=\rho(t,x)>0$,
$U={}^t(u_1(t,x),\ldots,u_n(t,x))$,
and $\theta=\theta(t,x)>0$ denote the mass density, the fluid velocity,
and the absolute temperature, respectively.
We adopt the constitutive relations for an ideal polytropic gas
\[
e=c_V\theta,\qquad p=R\rho\theta,
\]
where $c_V>0$ is the specific heat at constant volume and $R>0$ is the gas constant. 
The viscous stress tensor $D$, the strain tensor $\Theta$, and the heat flux $q$ are given by
\begin{align}
D= 2\nu\,\Theta+\lambda\,(\nabla\cdot U)I,\quad \Theta= \frac{1}{2}\{\nabla\otimes U + {}^t(\nabla\otimes U)\},\quad q=-\kappa \nabla\theta. 
\end{align}
The constants $\nu$ and $\lambda$ denote the shear and second viscosity
coefficients, respectively, while $\kappa$ is the thermal conductivity.
Throughout the paper, we assume
\begin{align}
\label{visc}
\nu>0, \quad 2\nu+ n\lambda \ge 0, \quad \kappa>0.
\end{align}
We restrict our attention to radially symmetric flows.
Accordingly, we look for solutions of the form
\begin{align}
\label{t}
\rho(t,x)=\rho(t,r), \quad U(t,x)=\frac{x}{r}\,u(t, r),\quad \theta(t,x)=\theta(t,r), \quad r= |x|,
\end{align}
where $u=u(t,r)$ is the radial component of the velocity.
Under this symmetry assumption, the system (\ref{cns}) reduces to
the following equations for $(\rho,u,\theta)$:
\begin{equation}
\label{nr}
\left\{\begin{array}{l}
\rho_t + \dfrac{ (r^{n-1}\rho u)_r}{r^{n-1}}=0, \\[10pt]
(\rho u)_t
+ \dfrac{ (r^{n-1}\rho u^2)_r}{r^{n-1}}+ p_r =\mu \Big(\dfrac{(r^{n-1}u)_r}{r^{n-1}}\Big)_r, \\[10pt]
\left\{\rho \left(c_V\theta + \frac{1}{2}u^2\right)\right\}_t +
\dfrac{\left\{r^{n-1}\rho u \left(c_V\theta + \frac{1}{2}u^2\right)\right\}_r}{r^{n-1}}
+\dfrac{ (r^{n-1}pu)_r}{r^{n-1}}
\\
\hspace{3cm}= \kappa\dfrac{(r^{n-1}\theta_r)_r}{r^{n-1}} +
\mu \Big(\dfrac{(r^{n-1}u)_r}{r^{n-1}}\Big)_ru +\Psi, \quad \ \ t>0,\ r>r_0,
\\[10pt]
\Psi := 2\nu u_r^2 +2\nu(n-1)(\dfrac{u}{r})^2+\lambda(u_r+\dfrac{n-1}{r}u)^2,
\end{array}
 \right.\,
\end{equation}
where 
\begin{align}
\label{viscondi0}
\mu := 2\nu + \lambda>0,
\end{align} 
and $\Psi$ represents the viscous dissipation.
Under the condition (\ref{visc}), it is a non-negative quadratic form
in $u_r$ and $u/r$. 
Thus, the energy equation involves both heat conduction and viscous heating. 
Moreover, (\ref{visc}) and (\ref{viscondi0}) imply
\begin{align}
\label{mutonu}
\frac{\nu}{\mu}\le \frac{n}{2(n-1)} \le 1.
\end{align}
To formulate the stationary problem and to specify the boundary data
associated with inflow and outflow, we first recall the corresponding
time-dependent initial-boundary value problem.
We prescribe the initial data
\begin{equation}
(\rho, u, \theta)(0,r)=(\rho_0, u_0, \theta_0)(r),\quad r >r_0, 
\end{equation}
\noindent
together with the far-field condition
\begin{equation}
\label{fc}
\displaystyle{\lim_{r \to \infty}}(\rho, u, \theta)(t,r)=(\rho_+,u_+,\theta_+),\quad t>0.
\end{equation}
At the inner boundary $r=r_0$, the prescribed data depend on the direction
of the flow. We impose
\begin{equation}
\label{bc}
 \left\{
\begin{array}{ll}
(\rho, u, \theta)(t,r_0) =(\rho_b, u_b, \theta_b),\quad t>0,& (u_b>0),\\[10pt]
(u, \theta)(t,r_0) = (u_b, \theta_b),\quad t>0,& (u_b\le 0).
\end{array}
 \right.
\end{equation}
Here $\rho_b,\rho_+,\theta_b,\theta_+>0$, $u_b$ and $u_+$ are given
constants.
According to the sign of $u_b$, we refer to the cases $u_b>0$, $u_b=0$,
and $u_b<0$ as the inflow problem, the impermeable-wall problem,
and the outflow problem, respectively.
Notice that in the inflow case the density, velocity, and temperature are
all prescribed at the boundary, whereas in the outflow case only the
velocity and temperature are imposed.
This difference in the boundary data will play an important role in the
inviscid-limit analysis below. 

The existence of stationary solutions corresponding to these boundary
conditions was established in our previous work \cite{H-M1}.
More precisely, for both the inflow and outflow problems, a unique
radially symmetric stationary solution was constructed in a sufficiently
small neighborhood of the far-field state.
For the outflow problem, the asymptotic stability of such stationary
solutions under small perturbations was recently proved in
\cite{NSH}.

We now formulate the stationary problem that will be studied throughout
the present paper.
Since a stationary solution is independent of $t$, the system
(\ref{nr}) together with the far-field and boundary conditions reduces to
%
\begin{equation}
\label{nrs0}
\left\{\begin{array}{l}
\dfrac{ (r^{n-1}\rho u)_r}{r^{n-1}}=0, \\[10pt]
\dfrac{ (r^{n-1}\rho u^2)_r}{r^{n-1}}+ p_r =\mu \Big(\dfrac{(r^{n-1}u)_r}{r^{n-1}}\Big)_r,
\\[10pt]
\dfrac{\left\{r^{n-1}\rho u \left(e + \frac{1}{2}u^2\right)\right\}_r}{r^{n-1}}
+\dfrac{ (r^{n-1}pu)_r}{r^{n-1}}\\[10pt]
\hspace{4cm}\!\!= \kappa\dfrac{(r^{n-1}\theta_r)_r}{r^{n-1}} +
\mu \Big(\dfrac{(r^{n-1}u)_r}{r^{n-1}}\Big)_ru +\Psi,\quad r\ge r_0,
\\[10pt]
p= R\rho\theta, \quad e=c_V \theta,\\
\Psi := 2\nu u_r^2 +2\nu(n-1)(\dfrac{u}{r})^2+\lambda(u_r+\dfrac{n-1}{r}u)^2,\\[10pt]
\displaystyle{\lim_{r\to \infty}} (\rho, u, \theta)(r) =(\rho_+, u_+, \theta_+), \\[2mm]
(\rho, u, \theta)(r_0)=(\rho_b, u_b, \theta_b)\ \ (u_b>0), \quad (u, \theta)(r_0)=(u_b, \theta_b) \ \ (u_b \le 0).
\end{array}
 \right.\,
\end{equation}
An important consequence of the stationary mass conservation law is
that the radial mass flux is constant.
Indeed, the first equation of (\ref{nrs0}) yields
\begin{align}
\label{n1}
r^{n-1}\rho(r) u(r)=\epsilon, \qquad r\ge r_0, 
\end{align}
for some constant $\epsilon$.
Its value is determined at the boundary according to the direction of
the flow:
\begin{equation}
\label{Nep}
\epsilon = r^{n-1}_0\rho_b u_b\quad(u_b>0),\qquad 
\epsilon = r^{n-1}_0\rho (r_0) u_b\quad(u_b\le0).
\end{equation}
This conservation law also restricts the admissible far-field velocity.
Since $\rho(r)\to\rho_+>0$ as $r\to\infty$, (\ref{n1}) gives, for
$n\ge2$,
\begin{align*}
u_+=\lim_{r\to \infty}u(r)= \lim_{r\to \infty}\frac{\epsilon}{r^{n-1}\rho(r)}=0.
\end{align*}
Thus, the zero far-field velocity is not an additional restriction but
is forced by radial mass conservation for stationary flows in the
exterior domain. Accordingly, we assume $u_+=0$ throughout the paper.
 

We next recall the existence and spatial decay properties of the
stationary solutions established in \cite{H-M1}, which form the starting
point for the inviscid-limit analysis developed below.
\noindent
\begin{thm}\rm{(\cite{H-M1})}
\label{mt}
\quad
Let $n \ge 3$ and $u_+=0$. 
Then, for any $\rho_+>0$ and $\theta_+>0$,
there exist positive constants $\epsilon_0$ and $C$ satisfying the following:

\medskip

\noindent
{\rm (I)} \ Let $u_b>0$. If $|u_b|+|\rho_b-\rho_+|+|\theta_b-\theta_+| \le \epsilon_0$, 
there exists a unique smooth solution $(\rho,u,\theta)$ of 
the problem {\rm (\ref{nrs0})} satisfying 
\begin{align*}
\begin{aligned}
  & |\rho(r)-\rho_+|,\ |\theta(r)-\theta_+|\le C r^{-(n-2)}(|u_b|+|\rho_b-\rho_+|+|\theta_b-\theta_+|),\\[5pt]
  & C^{-1}r^{-(n-1)}|u_b| \le |u(r)| \le C r^{-(n-1)}|u_b|, \qquad r\ge r_0.
\end{aligned}
\end{align*}\\
{\rm (II)} \ Let $u_b\le 0$. If $|u_b|+|\theta_b-\theta_+|\le \epsilon_0$, 
there exists a unique smooth solution $(\rho,u,\theta)$ 
of the problem {\rm (\ref{nrs0})} satisfying
\begin{align*}
\begin{aligned}
  & |\rho(r)-\rho_+|,\ |\theta(r)-\theta_+|\le C r^{-(n-2)}(|u_b|+|\theta_b-\theta_+|), \\[5pt]
  & C^{-1}r^{-(n-1)}|u_b| \le |u(r)|\le C r^{-(n-1)}|u_b|, \qquad r\ge r_0.
\end{aligned}
\end{align*}
\end{thm}
\noindent
As observed in \cite{H-M1}, when the viscosity parameter is restricted
to a bounded interval, say $\mu\in(0,1]$, the constants $\epsilon_0$
and $C$ in Theorem \ref{mt} can be chosen independently of $\mu$.
Moreover, the constant $C$ can be taken independently of the smallness
constant $\epsilon_0$.
These uniform estimates will be used throughout the zero-viscosity
analysis below.

Our main purpose is to determine the behavior of these stationary
solutions as the viscosity tends to zero.
A distinctive feature of the present heat-conductive model is that,
while the viscous terms disappear in the limit $\mu\to0$, the heat
conduction remains present.
Accordingly, the limiting problem consists of the inviscid
mass and momentum equations coupled with a heat-conductive energy
equation:
\begin{equation}
\label{Es}
\left\{\begin{array}{l}
\dfrac{ (r^{n-1}\rho^E u^E)_r}{r^{n-1}}=0, \\[10pt]
\dfrac{ (r^{n-1}\rho^E (u^E)^2)_r}{r^{n-1}}+ p^E_r =0,\\[10pt]
\dfrac{\left\{r^{n-1}\rho^E u^E \left(e^E + \frac{1}{2}(u^E)^2\right)\right\}_r}{r^{n-1}}
+\dfrac{ (r^{n-1}p^Eu^E)_r}{r^{n-1}}= \kappa\dfrac{(r^{n-1}\theta^E_r)_r}{r^{n-1}},\quad r\ge r_0,\\[10pt]
\displaystyle{\lim_{r\to \infty}} (\rho^{E}, u^{E}, \theta^{E})(r) =(\rho_+, u_+, \theta_+),\qquad \theta^E(r_0)=\theta_b,
\\[10pt]
p^E(\rho^E, \theta^E):= R\rho^E\theta^E, \quad e^E:=c_V \theta^E.
\end{array}
 \right.\,
\end{equation}
In addition to the temperature condition
$\theta^E(r_0)=\theta_b$, an appropriate boundary condition for
the pair $(\rho^E,u^E)$ must be selected according to the direction
of the flow.
The reason is that, in the subsonic regime considered here, the
mass-momentum part of the limiting system admits only one boundary
condition at $r=r_0$.
For the outflow problem, the natural choice is
\[
u^E(r_0)=u_b.
\]
We shall prove that, with this boundary condition, the viscous
stationary solution converges uniformly to the corresponding solution
of the limiting system as $\mu\to0$.
In particular, no boundary-layer correction is required in the
outflow case.

The inflow problem has a different structure.
In this case, the boundary data for the viscous problem prescribe both
the density and the velocity at $r=r_0$, whereas only one boundary
condition can be imposed on the mass-momentum part of the limiting
system.
This mismatch gives rise to a boundary layer near the inner boundary.

The stationary mass conservation law provides a natural way to select
the boundary condition for the limiting flow.
Indeed, both the viscous and limiting solutions satisfy a constant
radial mass-flux relation.
Since the viscous mass flux is fixed by the inflow data as
$r_0^{\,n-1}\rho_bu_b$, we impose
\[
(\rho^E u^E)(r_0)=\rho_bu_b
\]
on the limiting system.

The remaining discrepancy between the viscous boundary data and the
boundary value of the limiting solution is described by a boundary
layer.
To identify its leading-order profile, we introduce the stretched
variable
\[
y:=\frac{r-r_0}{\mu}
\]
and rewrite the stationary equations in terms of $y$.
Formally passing to the leading-order equations as $\mu\to0$ yields
the boundary-layer problem

%
\begin{align}
\label{Ebl} 
\left\{
\begin{aligned}
  &\ ({\rho}^B{u}^B)_y = 0,   \\
  &\ ({\rho}^B (u^B)^2)_y+p^B_y={u}^B_{yy}, \qquad y>0,\\[2mm]
  &\  \kappa \theta^B_{yy}=0,\\
  &\ \displaystyle{\lim_{y\to \infty}} ({\rho}^B,{u}^B, \theta^B)(y) =(\rho^{E}(r_0), u^{E}(r_0), \theta_b), \\[2mm]
  & ({\rho}^B,{u}^B, \theta^B)(0)=(\rho_b, u_b, \theta_b),
\end{aligned}
 \right.\,
\end{align}
where $p^B=R\rho^B\theta^B$. An important feature of this boundary-layer system is that the
temperature profile is constant:
\[
\theta^B(y)\equiv\theta_b.
\]
Thus, the nontrivial boundary-layer corrections occur only in the
density and velocity, while the temperature is already matched through
the heat-conductive limiting problem.

We shall show that, in the inflow case, the viscous stationary solution
is uniformly approximated by the superposition of the limiting
stationary solution and this boundary-layer profile.

The inviscid limit of the compressible Navier--Stokes equations has been
studied extensively.
We refer, for example, to \cite{E, Hu, HL,S,WW,XY} and the references therein. However, there exist no results on the inflow and outflow problems for the radially symmetric stationary solutions in exterior domains, except our paper \cite{H-M}.  
Our previous work \cite{H-M} treated the corresponding inviscid problems for a 
barotropic viscous gas. 
The present paper extends that analysis to a viscous and heat-conductive
polytropic ideal gas.  

%
%
%

This extension changes the structure of the limiting problem in an
essential way.
While the viscous terms disappear as $\mu\to0$, the heat-conduction
term remains in the energy equation.
For the inflow problem, the mismatch in the mass-momentum boundary data
produces nontrivial boundary-layer corrections in the density and
velocity, whereas no boundary layer appears for the temperature. 
Thus, the limiting profile consists of the stationary solution of the
heat-conductive inviscid system together with boundary-layer corrections
only in the density and velocity. 
We are now ready to state the main results.
We first consider the outflow problem.

\begin{thm}\label{etE}
Let $n\geq 3$, $u_b\leq 0$, and $u_+=0$.
Then, for any $\rho_+>0$ and $\theta_+>0$, there exist positive
constants $\epsilon_0$ and $C$, independent of $\mu\in(0,1]$,
such that the following statements hold.

\smallskip
\noindent
{\rm (I)}
If $
 |u_b|+|\theta_b-\theta_+|\leq\epsilon_0,
$ 
then there exists a unique smooth solution
$(\rho^E,u^E,\theta^E)$ of the problem (\ref{Es})
with the boundary conditions 
$
 u^E(r_0)=u_b,\ \theta^E(r_0)=\theta_b,
$ 
satisfying
\[
 |\rho^E(r)-\rho_+|,\ |\theta^E(r)-\theta_+|
 \leq
 Cr^{-(n-2)}
 \bigl(|u_b|+|\theta_b-\theta_+|\bigr),
\]
and
\[
 C^{-1}r^{-(n-1)}|u_b|
 \leq |u^E(r)|
 \leq Cr^{-(n-1)}|u_b|,
 \qquad r\geq r_0.
\]

\bigskip
\noindent
{\rm (II)}
If 
$
 |u_b|+|\theta_b-\theta_+|\leq\epsilon_0,
$ 
then the solution $(\rho,u,\theta)$ of the problem (\ref{nrs0})
uniformly converges to the solution
$(\rho^E,u^E,\theta^E)$ of the problem (\ref{Es})
as $\mu\to0$. More precisely,
\[
 \sup_{r\geq r_0}
 \left|
 (\rho-\rho^E,u-u^E,\theta-\theta^E)(r)
 \right|
 \leq C\mu.
\]
\end{thm}

%
%

\noindent
We next consider the inflow problem. 

\begin{thm}
\label{bl+E}
Let $n\geq 3$, $u_b>0$, and $u_+=0$.
Then, for any $\rho_+>0$ and $\theta_+>0$, there exist positive
constants $\epsilon_0$, $\nu_0$, and $C$, independent of
$\mu\in(0,1]$, such that the following statements hold.

\smallskip
\noindent
{\rm (I)}
If 
$
 |u_b|+|\rho_b-\rho_+|+|\theta_b-\theta_+|
 \leq\epsilon_0,
$ 
then there exists a unique smooth solution
$(\rho^E,u^E,\theta^E)$ of the problem (\ref{Es})
with the boundary conditions 
$
 \rho^E(r_0)u^E(r_0)=\rho_bu_b,
 \ 
 \theta^E(r_0)=\theta_b,
$ 
satisfying
\[
 |\rho^E(r)-\rho_+|,\ |\theta^E(r)-\theta_+|
 \leq
 Cr^{-(n-2)}
 \bigl(|u_b|+|\theta_b-\theta_+|\bigr),
\]
and
\[
 C^{-1}r^{-(n-1)}|u_b|
 \leq |u^E(r)|
 \leq Cr^{-(n-1)}|u_b|,
 \qquad r\geq r_0.
\]

\bigskip
\noindent
{\rm (II)}
If 
$
 |u_b|+|\rho_b-\rho_+|+|\theta_b-\theta_+|
 \leq\epsilon_0,
$ 
then there exists a unique smooth boundary-layer solution
$(\rho^B,u^B,\theta^B)$ of the problem (\ref{Ebl}) satisfying
\[
 |\rho^B(y)-\rho^E(r_0)|
 \leq
 C\exp\left(-\frac{\nu_0}{\epsilon}y\right)
 \bigl(
 |u_b|+|\rho_b-\rho_+|+|\theta_b-\theta_+|
 \bigr),
\]
\[
 |u^B(y)-u^E(r_0)|
 \leq
 C\exp\left(-\frac{\nu_0}{\epsilon}y\right)
 \bigl(
 |u_b|+|\rho_b-\rho_+|+|\theta_b-\theta_+|
 \bigr),
\]
and
\[
 \theta^B(y)\equiv\theta_b,
 \qquad y\geq0,
\]
where $\epsilon=\rho_bu_b$.

\bigskip

\noindent
{\rm (III)}
If 
$
 |u_b|+|\rho_b-\rho_+|+|\theta_b-\theta_+|
 \leq\epsilon_0,
$ 
then the solution $(\rho,u,\theta)$ of the problem (\ref{nrs0})
is uniformly approximated, as $\mu\to0$, by the superposition of
the limiting stationary solution $(\rho^E,u^E,\theta^E)$
and the boundary-layer solution $(\rho^B,u^B,\theta^B)$.
More precisely, for any $\alpha\in(0,1)$,
\[
 \sup_{r\geq r_0}
 \left|
 \rho(r)-\rho^E(r)
 -\rho^B\left(\frac{r-r_0}{\mu}\right)
 +\rho^E(r_0)
 \right|
 \leq C\mu^\alpha,
\]
\[
 \sup_{r\geq r_0}
 \left|
 u(r)-u^E(r)
 -u^B\left(\frac{r-r_0}{\mu}\right)
 +u^E(r_0)
 \right|
 \leq C\mu^\alpha,
\]
and
\[
 \sup_{r\geq r_0}
 |\theta(r)-\theta^E(r)|
 \leq C\mu^\alpha,
\]
where $C$ may also depend on $\alpha$.
\end{thm}

This paper is organized as follows.
In Section 2, we reformulate the stationary problem (\ref{nrs0})
in terms of the specific volume and the perturbations from the far-field
state, and recall the existence and decay estimates that are uniform
with respect to $\mu\in(0,1]$.
Section 3 is devoted to the outflow problem.
We construct the corresponding solution of the limiting problem
(\ref{Es}) and prove that the viscous stationary solution converges
uniformly to it as $\mu\to0$.
In Section 4, we treat the inflow problem.
We construct both the limiting stationary solution and the boundary-layer
profile, and show that the viscous stationary solution is uniformly
approximated by their superposition, with an algebraic rate of convergence.


\begin{remark}
The exponential estimate in Theorem \ref{bl+E} (II) shows that the boundary
layer is localized near the inner boundary on the scale $O(\mu)$.
Indeed, since
\[
    y=\frac{r-r_0}{\mu},
\]
the decay factor can be written as
\[
    \exp\left\{-\frac{\nu_0(r-r_0)}{\epsilon\mu}\right\}.
\]
Hence the boundary-layer correction becomes exponentially small away
from an $O(\mu)$-neighborhood of $r=r_0$.
\end{remark}


\begin{remark}
In the present analysis, the viscosity parameter $\mu$ tends to zero
while the heat-conductivity coefficient $\kappa$ is kept fixed.
It would be interesting to study the simultaneous limit
$\mu\to0$ and $\kappa\to0$.
In that case, the structure of the limiting energy equation and the
corresponding thermal boundary layer may be essentially different
from those considered here.
\end{remark}

\section{Reformulation of the stationary problem}

We first rewrite the stationary problem \eqref{nrs0} in a form
that will be convenient for the subsequent analysis.
Throughout this section, we restrict ourselves to the case $u_b\ne0$.
Indeed, when $u_b=0$, equivalently $\epsilon=0$ in \eqref{Nep},
the stationary problem can be solved explicitly as
\[
 (\rho,u,\theta)
 =
 \left(
 \frac{\rho_+\theta_+}{\widetilde\theta},
 0,
 \widetilde\theta
 \right),
 \qquad
 \widetilde\theta(r)
 =
 \theta_+
 +(\theta_b-\theta_+)
 \left(\frac{r}{r_0}\right)^{-(n-2)}.
\]
Hence, only the case $u_b\ne0$ needs to be considered below.

For simplicity, we normalize the inner radius by setting $r_0=1$.
We then introduce the specific volume
\[
 v=\frac{1}{\rho},
\]
and write $v_b=1/\rho_b$ and $v_+=1/\rho_+$.
By the mass conservation relation \eqref{n1}, the velocity is expressed as
%
%

\begin{align}
\label{urv}
u(r)=\frac{\epsilon}{r^{n-1}}v(r),\quad r\ge 1,
\end{align}
where $\epsilon = u_b/v_b\ (u_b>0)$, and
$\epsilon = u_b/v(1)\ (u_b <0)$. 
Substituting (\ref{urv}) into the second equation of (\ref{nrs0}), we have
\begin{align}
\label{base}
\frac{\epsilon^2}{r^{n-1}}\left(\frac{v}{r^{n-1}}\right)_r+\tilde{p}(v, \theta)_r
= \epsilon\mu\left(\frac{v_r}{r^{n-1}}\right)_r,
\end{align}
where
\[
  \widetilde p(v,\theta)
  :=p\left(\frac{1}{v},\theta\right)
  =\frac{R\theta}{v}.
\]
In particular,
\[
  \widetilde p_v(v,\theta)
  =-\frac{R\theta}{v^2}<0
  \qquad (v>0,\ \theta>0).
\]
We further introduce new unknown functions $\eta$ and $\chi$
as the deviations of $v$ and $\theta$ from their far-field values
$v_+$ and $\theta_+$, respectively, by
\begin{equation}
 \eta(r)=v(r)-v_+,\qquad
 \chi(r)=\theta(r)-\theta_+,\qquad r\ge1.
 \label{transec}
\end{equation}
Substituting the variables defined in \eqref{transec} into \eqref{base}, we obtain 
\begin{equation}
\label{Narrange}
\epsilon \mu \left(\frac{ \eta_r}{r^{n-1}}\right)_r=
\tilde p(v_++\eta, \theta_++\chi)_r
+\frac{\epsilon^2 v_+}{2}\Big(\frac{1}{r^{2(n-1)}}\Big)_r
+\frac{\epsilon^2}{r^{n-1}}\left(\frac{\eta}{r^{n-1}}\right)_r,
\end{equation}
where $\epsilon = u_b/v_b\ (u_b>0)$, and $\epsilon = u_b/(v_++\eta(1))\ (u_b < 0)$. 
Equation \eqref{Narrange} can be written equivalently as
\begin{equation}
\label{Narranger}
\left(\epsilon \mu \frac{\eta_r}{r^{n-1}}-\tilde p(v_++\eta, \theta_++\chi) 
-\frac{\epsilon^2 v_+}{2r^{2(n-1)}}
-\frac{\epsilon^2\eta}{r^{2(n-1)}}
+\epsilon^2(n-1)\int_r^{\infty}\frac{\eta(s)}{s^{2n-1}}ds
\right)_r =0.
\end{equation}

Hence, the quantity in parentheses in \eqref{Narranger} is independent of $r$.
Using
\[
  \eta(r)\to0,\qquad \chi(r)\to0
  \quad\text{as }r\to\infty,
\]
we find that this constant is $-\widetilde p(v_+,\theta_+)$. 
Consequently, the equation for \(\eta\) can be written as
\begin{align}
\label{eta_r}
\begin{aligned}
&\eta_r= 
\frac{r^{n-1}}{\epsilon\mu}\big(\tilde p(v_++\eta, \theta_++\chi)- \tilde p(v_+,\theta_+)\big)\\[2mm]
&\qquad\ +\frac{\epsilon v_+}{2\mu}\frac{1}{r^{n-1}}
+\frac{\epsilon\eta}{\mu r^{n-1}}
-\frac{\epsilon(n-1)r^{n-1}}{\mu}\int_r^{\infty}\frac{\eta(s)}{s^{2n-1}}ds,
\quad r>1.
 \end{aligned}
\end{align}
To rewrite the third equation of \eqref{nrs0}, we multiply it by \(r^{n-1}\) and substitute \(u=\epsilon(v_++\eta)/r^{n-1}\). This yields
\footnotesize
\begin{align}
\label{enea}
\begin{aligned}
&\left\{\epsilon \left(c_V\theta + \frac{1}{2}\left(\frac{\epsilon (v_++\eta)}{r^{n-1}}\right)^2\right)\right\}_r
+(R\epsilon\theta)_r\\
&\hspace{3cm}= \kappa(r^{n-1}\theta_r)_r+\epsilon^2\mu \Big(\dfrac{(v_++\eta)\eta_r}{r^{n-1}}\Big)_r
+\left(\frac{-2\e^2\nu(n-1)(v_++\eta)^2}{r^n}\right)_r.
\end{aligned}
\end{align}
\normalsize
Substituting \(\theta=\theta_++\chi\) into \eqref{enea} and integrating the resulting equation over \([r,\infty)\), we obtain
\begin{align*}
\epsilon c_P\chi + \dfrac{\epsilon^3}{2}\Big(\dfrac{v_++\eta}{r^{n-1}}\Big)^2
- \kappa r^{n-1}\chi_r = \alpha
+\epsilon^2\mu \dfrac{(v_++\eta)\eta_r}{r^{n-1}}
-\frac{2\e^2\nu(n-1)(v_++\eta)^2}{r^n},
\end{align*}
where $c_P:=R+c_V$ is the specific heat at constant pressure, and $\alpha$ is a constant of integration.
\\
Accordingly, the stationary problem \eqref{nrs0} can be reformulated in terms of \((\eta,\chi)\) as follows:
\begin{align}
\label{reforme}
 \left\{
\begin{aligned}
&\eta_r= 
\frac{r^{n-1}}{\epsilon\mu}\big(\tilde p(v_++\eta, \theta_++\chi)- \tilde p(v_+,\theta_+)\big)\\[2mm]
&\qquad\ +\frac{\epsilon v_+}{2\mu}\frac{1}{r^{n-1}}
+\frac{\epsilon\eta}{\mu r^{n-1}}
-\frac{\epsilon(n-1)r^{n-1}}{\mu}\int_r^{\infty}\frac{\eta(s)}{s^{2n-1}}ds,
\quad r>1,\\[2mm]
&\chi_r=-\frac{\alpha}{\kappa r^{n-1}}+ H_0(\eta,\eta_r,\chi)(r), \quad r>1,\\[2mm]
&H_0(\eta,\eta_r,\chi)(r):=\frac{\epsilon c_P}{\kappa r^{n-1}}\chi + 
\frac{\epsilon^3}{2}\frac{(v_++\eta)^2}{\kappa r^{3(n-1)}}\\
&\hspace{5cm}-\epsilon^2\mu \dfrac{(v_++\eta)\eta_r}{\kappa r^{2(n-1)}}+\frac{2\epsilon^2\nu(n-1)(v_++\eta)^2}{\kappa r^{2n-1}},\\
&\displaystyle{\lim_{r\to \infty}}\eta(r) =0,\quad \displaystyle{\lim_{r\to \infty}}\chi(r) =0,  \\[2mm]
&\eta(1):(no\ boundary\ condition)\  (u_b<0), \ \ \eta(1)= \eta_b:=v_b-v_+\  (u_b>0),\\[2mm] 
&\chi(1)= \chi_b:=\theta_b-\theta_+,\\
&\epsilon = u_b/(v_++\eta(1))<0\  (u_b<0), \qquad \epsilon = u_b/(v_++\eta_b)>0\  (u_b>0),
 \end{aligned}
 \right.
\end{align}
where we define
$H(\eta,\eta_r,\chi)(r):= -\int_r^\infty H_0(\eta,\eta_r,\chi)(\tau)\ d\tau$
and $\alpha$ is defined
\begin{align}
\label{alpha}
\alpha := \kappa(n-2)\chi_b-\kappa(n-2)H(\eta,\eta_r,\chi)(1),
\end{align}
so that $\chi(1)=\chi_b$.
Once the desired solution $(\eta, \chi)$ of (\ref{reforme}) is obtained, 
the velocity $u$ is immediately obtained by (\ref{urv}) and (\ref{transec}) as
\begin{equation*}
u(r)=\frac{u_b(v_++\eta(r))}{v_br^{n-1}}\quad (u_b>0),
\quad
u(r)=\frac{u_b(v_++\eta(r))}{(v_++\eta(1))r^{n-1}}\quad (u_b < 0).
\end{equation*} 
We recall the following existence result for the reformulated problem \(\eqref{reforme}\). 
\noindent
\begin{thm}{\rm(\cite{H-M1})}
\label{rmt}
Let $n\geq 3$. Then, for any $v_+>0$ and $\theta_+>0$,
there exist positive constants $\epsilon_0$ and $C$,
independent of $\mu\in(0,1]$, such that the following statements hold:
\medskip

\noindent
{\rm (I)} \ Let $u_b>0$. If $|u_b|+|\eta_b|+|\chi_b|\le \epsilon_0$, 
there exists a unique smooth solution $(\eta,\chi)$ of the problem {\rm (\ref{reforme})} satisfying 
\begin{align}
\label{decayi}
|\eta(r)|,\ \ |\chi(r)|\le Cr^{-(n-2)}(|u_b|+|\eta_b|+|\chi_b|),\qquad  r\ge 1. 
\end{align}
{\rm (II)} \ Let $u_b < 0$. If $|u_b|+|\chi_b| \le \epsilon_0$, 
there exists a unique smooth solution $(\eta,\chi)$ of the problem {\rm (\ref{reforme})} satisfying
\begin{align}
\label{decayo}
|\eta(r)|,\ \ |\chi(r)|\le Cr^{-(n-2)}(|u_b|+|\chi_b|), \qquad r\ge 1. 
\end{align}
\end{thm}
Theorem \ref{rmt} was proved in \cite{H-M1} by constructing approximate sequences and applying the contraction mapping theorem.


\section{Outflow problem}
This section is devoted to the outflow case $u_b<0$ and to the proof
of Theorem~\ref{etE}(II). More precisely, we compare the stationary
solution of \eqref{nrs0} with the corresponding solution of the
inviscid problem \eqref{Es} satisfying
\[
 u^E(1)=u_b, \qquad \theta^E(1)=\theta_b,
\]
and establish their convergence as $\mu\to0$.
\subsection{Inviscid problem}

We begin by studying the inviscid problem \eqref{Es} subject to
\[
 u^E(1)=u_b, \qquad \theta^E(1)=\theta_b.
\]
We first derive a convenient formulation of the problem and then
establish the existence of its solution. 
The first equation of \eqref{Es} implies that
\begin{align}
\label{n1e}
r^{n-1}\rho^{E}(r) u^{E}(r)=\epsilon^{E}, \qquad r\ge 1, 
\end{align}
for some constant \(\epsilon^E\). Using the boundary condition \(u^E(1)=u_b\), we obtain
\begin{equation*}
\epsilon^{E} =  u_b \rho^{E}(1)= u_b/v^{E}(1),
\end{equation*}
where $v^{E}(r):=1/\rho^{E}(r)$. 
It then follows from (\ref{n1e}) that the velocity \(u^E\) is given by
\begin{align}
\label{uE}
u^{E}(r)=\frac{\epsilon^{E}}{r^{n-1}}v^{E}(r),\quad r\ge 1.
\end{align}
Substituting (\ref{uE}) into the second equation of (\ref{Es}), and introducing
\[
\eta^E(r):=v^E(r)-v_+,\qquad
\chi^E(r):=\theta^E(r)-\theta_+,
\qquad
\chi_b:=\theta_b-\theta_+,
\]
we obtain
\begin{align*}
\frac{|\epsilon^{E}|^2}{r^{n-1}}\left(\frac{v_++\eta^{E}}{r^{n-1}}\right)_r+\tilde{p}(v_++\eta^{E}, \theta_++\chi^{E})_r=0. 
\end{align*}
This equation can also be rewritten as
\begin{align}
\label{E_r2}
\eta^{E}_r(r)=\frac{r^{-(2n-1)}({\epsilon^E})^2(n-1)(v_++\eta^{E})-\ti{p}_{\theta}(v_++\eta^E, \theta_++\chi^E)\chi^E_r}{(\epsilon^E)^2r^{-2(n-1)}+\tilde p_v(v_++\eta^{E}, \theta_++\chi^E)}.
\end{align}
Integrating \(\eqref{E_r2}\) over \([r,\infty)\), we obtain 
\begin{align}
\label{E2}
\begin{aligned}
\eta^{E}(r)&=\int_r^{\infty} \frac{s^{-(2n-1)}|\epsilon^E|^2(n-1)(v_++\eta^{E})-\ti{p}_{\theta}(v_++\eta^E, \theta_++\chi^E)\chi^E_s}{|\tilde p_v(v_++\eta^{E}, \theta_++\chi^E)|-|\epsilon^E|^2s^{-2(n-1)}} ds \\
&=:\int_r^{\infty}G[\eta^{E}, \chi^{E}](s)ds,
\end{aligned}
\end{align}
where we have used the far-field condition
\[
\lim_{r\to\infty}\eta^E(r)=0.
\]
Following the same procedure used to derive (\ref{reforme}) from (\ref{urv}), we can also reformulate the second equation of (\ref{Es}) in another form:
\begin{align}
\label{reformE}
 \left\{
\begin{aligned}
&\big(\tilde p(v_++\eta^{E}, \theta_++\chi^E)- \tilde p(v_+, \theta_+)\big)+\frac{|\epsilon^{E}|^2 v_+}{2}\frac{1}{r^{2(n-1)}}\\[2mm]
&\qquad\qquad\qquad\qquad\quad
+\frac{|\epsilon^{E}|^2\eta^{E}}{r^{2(n-1)}}-|\epsilon^{E}|^2(n-1)\int_r^{\infty}\frac{\eta^{E}(s)}{s^{2n-1}}ds=0,
\quad r \ge 1,
\\[2mm]
&\displaystyle{\lim_{r\to \infty}}\eta^{E}(r) =0,\quad  \displaystyle{\lim_{r\to \infty}}\chi^{E}(r) =0,\quad \chi^E(1)=\chi_b,
\end{aligned}
 \right.
\end{align}
where \(\epsilon^E=u_b/(v_++\eta^E(1))\). Using \(\widetilde p(v,\theta)=R\theta/v\) in (\ref{reformE}), we obtain
\begin{equation}
\label{reformE-1}
R(\dfrac{\theta_++\chi^{E}}{v_++\eta^{E}}-\dfrac{\theta_+}{v_+})
+\frac{|\epsilon^{E}|^2 v_+}{2r^{2(n-1)}}
+\frac{|\epsilon^{E}|^2\eta^{E}}{r^{2(n-1)}}
-|\epsilon^{E}|^2(n-1)\int_r^{\infty}\frac{\eta^{E}(s)}{s^{2n-1}}ds
=0. 
\end{equation}
Noting that the first term of (\ref{reformE-1}) can be rewritten as 
\begin{equation}
R(\dfrac{\theta_++\chi^{E}}{v_++\eta^{E}}-\dfrac{\theta_+}{v_+})
=\dfrac{R}{v_+}\chi^{E}-\dfrac{R\theta_+}{v_+^2}\eta^{E}-\frac{R\chi^{E}\eta^{E}}{v_+(v_++\eta^{E})}+\frac{R\theta_+(\eta^{E})^2}{v_+^2(v_++\eta^{E})}, 
\end{equation}
we rewrite (\ref{reformE-1}) as  
\begin{align}
\label{reformE2}
\begin{aligned}
&\dfrac{R}{v_+}\chi^{E}-\dfrac{R\theta_+}{v_+^2}\eta^{E}-\frac{R\chi^{E}\eta^{E}}{v_+(v_++\eta^{E})}+\frac{R\theta_+(\eta^{E})^2}{v_+^2(v_++\eta^{E})}\\
&\qquad +\frac{|\epsilon^{E}|^2 v_+}{2r^{2(n-1)}}
+\frac{|\epsilon^{E}|^2\eta^{E}}{r^{2(n-1)}}
-|\epsilon^{E}|^2(n-1)\int_r^{\infty}\frac{\eta^{E}(s)}{s^{2n-1}}ds
=0. 
\end{aligned}
\end{align}
Solving (\ref{reformE2}) for \(\eta^E\), we obtain
\begin{align}
\label{ieEout}
\begin{aligned}
\eta^{E}=&\frac{v_+}{\theta_+}\chi^{E}-\frac{v_+\chi^{E}\eta^{E}}{\theta_+(v_++\eta^{E})}+\frac{(\eta^{E})^2}{v_++\eta^{E}}\\
&+\frac{|\epsilon^{E}|^2v_+^3}{2R\theta_+r^{2(n-1)}}+\frac{|\epsilon^{E}|^2v_+^2\eta^{E}}{R\theta_+r^{2(n-1)}}-\frac{|\epsilon^{E}|^2(n-1)v_+^2}{R\theta_+}\int_r^{\infty}\frac{\eta^{E}(s)}{s^{2n-1}}ds\\
&=:\dfrac{v_+}{\theta_+}\chi^{E} +F(\eta^{E},\chi^{E}).
\end{aligned}
\end{align}
We next rewrite the third equation of (\ref{Es}). 
Substituting (\ref{uE}) into the third equation of (\ref{Es}) and integrating with respect to \(r\), we obtain
\begin{align}
\label{chieout}
\epsilon^Ec_P\chi^E+\frac{(\epsilon^E)^3}{2}(\frac{v_++\eta^E}{r^{n-1}})^2-\kappa(r^{n-1}\chi^E_r)=\alpha^E,
\end{align}
where \(\alpha^E\) is an integration constant. 
Solving (\ref{chieout}) for \(\chi_r^E\), we obtain 
\begin{align}
\label{chi^E_r}
\chi^E_r=-\frac{\alpha^E}{\kappa r^{n-1}}+\frac{\epsilon^E c_P\chi^E}{\kappa r^{n-1}}+\frac{(\epsilon^E)^3}{2}\frac{(v_++\eta^E)^2}{\kappa r^{3(n-1)}}.
\end{align}
Integrating (\ref{chi^E_r}) over \([r,\infty)\), we obtain
\begin{align}
\begin{aligned}
\label{chi^Eintout}
\chi^E(r)&=\frac{\alpha^E}{\kappa (n-2)r^{n-2}}-
\int^\infty_r \Big\{\frac{\epsilon^E c_P\chi^E(s)}{\kappa s^{n-1}} + 
\dfrac{(\epsilon^E)^3(v_++\eta^E(s))^2}{2\kappa s^{3(n-1)}}\Big\}\,ds\\
&=: \frac{\alpha^E}{\kappa (n-2)r^{n-2}}+H^E(\eta^{E}, \chi^{E})(r).
\end{aligned}
\end{align}
Setting \(r=1\) and using \(\chi^E(1)=\chi_b\), we obtain
\begin{equation}
\label{alphadef}
\alpha^E = \kappa(n-2)\chi_b-
\kappa(n-2)H^E(\eta^{E}, \chi^{E})(1).
\end{equation}
The existence result to be proved for the reformulated problem (\ref{ieEout}) and (\ref{chi^Eintout}) is stated as follows.
\begin{thm} 
\label{rmtE}
Let $n \ge 3$. For any $v_+>0$ and $\theta_+>0$, there exist 
positive constants $\epsilon_0$ and $C$ such that if $|u_b|+|\chi_b| \le \epsilon_0$, 
the system consisting of (\ref{ieEout}) and (\ref{chi^Eintout}), with \(\alpha^E\) defined by (\ref{alphadef}), admits a unique smooth solution \((\eta^E,\chi^E)\) satisfying
\begin{subequations}
\begin{align}
\label{decayE1}
&|\eta^{E}(r)|,\ |\chi^{E}(r)|\le Cr^{-(n-2)}(|u_b|+|\chi_b|), \qquad r\ge 1, \\[2mm]
\label{decayE22}
&|\eta^{E}_r(r)|,\ |\chi^{E}_r(r)| \le Cr^{-(n-1)}(|u_b|+|\chi_b|), \ \ \qquad r\ge 1.
\end{align}
\end{subequations}
\end{thm}
Since \(\rho^E=(v_++\eta^E)^{-1}\), \(\theta^E=\theta_++\chi^E\), and \(u^E\) is given by (\ref{uE}), Theorem \ref{rmtE} immediately yields statement (I) of Theorem \ref{etE}.

To prove Theorem \ref{rmtE}, we seek a solution of the integral equations (\ref{ieEout}) and (\ref{chi^Eintout}) in the Banach space \(X\), equipped with the norm $\|\cdot\|_X$, defined by
\begin{align*}
X=\left\{f\in C([1,\infty)):\sup_{r\ge1}|r^{n-2}f(r)|<\infty\right\},
\quad \| f \|_X = \sup_{r\ge 1}|r^{n-2}f(r)|.
\end{align*}
\noindent
For this purpose, we define the approximate sequence \(\{(\eta^{E(m)},\chi^{E(m)})\}_{m\ge0}\) recursively by
\begin{align}
\label{appEout}
\left\{\begin{array}{l}
{\eta^E}^{(m+1)} = \dfrac{v_+}{\theta_+}{\chi^E}^{(m+1)} +F({\eta^E}^{(m)},{\chi^E}^{(m)}),\\[5pt]
{\chi^E}^{(m+1)} = \frac{\chi_b-H^E({\eta^E}^{(m)}, {\chi^E}^{(m)})(1)}{r^{n-2}}+H^E({\eta^E}^{(m)},{\chi^E}^{(m)}),\\[5pt]
{\epsilon^{E}}^{(m)}=u_b/(v_++{\eta^{E}}^{(m)}(1)),
\end{array}
 \right.\,
\end{align}
starting with $\eta^{E(0)}=\chi^{E(0)}=0$. 

\medskip

\noindent
{\it Proof of Theorem \ref{rmtE}.}\\
Since $|u_b|$ is assumed to be sufficiently small, we may assume that
$|u_b|<1$. 
We first establish a uniform bound for the approximate sequence. More precisely, there exist positive constants \(\delta\) and \(C\) such that, whenever
\[
|u_b|+|\chi_b|\le\delta,
\]
one can choose \(M>0\) satisfying
\begin{equation}
\label{eu}
\|{\eta^{E}}^{(m)}\|_X,\quad \|{\chi^{E}}^{(m)}\|_X \le M \le C(|u_b|+|\chi_b|),\quad m\ge 0.
\end{equation}
Here and below, \(C>0\) denotes a generic constant and \(\delta>0\) a sufficiently small constant. They are independent of \(u_b\) and \(\chi_b\), but may depend on \(v_+\), \(\theta_+\), \(n\), \(\kappa\), \(R\), and \(c_P\). 
We prove \eqref{eu} by mathematical induction. 

\smallskip

\noindent
\underline{Case $m=0$}. \quad
We recall that 
$
\eta^{E(0)}=\chi^{E(0)}=0.
$
Then, by the definition of $\epsilon^{E(m)}$, we have
\[
\epsilon^{E(0)}=\frac{u_b}{v_+}.
\]
Substituting $m=0$ into the second equation of \eqref{appEout}
and multiplying by $r^{n-2}$, we obtain
\[
\begin{aligned}
r^{n-2}\chi^{E(1)}(r)
={}&
\chi_b
-H^E(\eta^{E(0)},\chi^{E(0)})(1) +r^{n-2}H^E(\eta^{E(0)},\chi^{E(0)})(r)\\
=&\chi_b+
\int^\infty_1 
\dfrac{(\frac{u_b}{v_+})^3 v_+^2}{2\kappa s^{3(n-1)}}\,ds-r^{n-2}\int^\infty_r \dfrac{(\frac{u_b}{v_+})^3v_+^2}{2\kappa s^{3(n-1)}}\,ds\\
=&\chi_b+\frac{u_b^3}{2\kappa v_+(3n-4)}\left(1-\frac{1}{r^{2(n-1)}}\right).
\end{aligned}
\]
Hence,
\begin{align}
\label{appEout01}
\begin{aligned}
|r^{n-2}{\chi^E}^{(1)}(r)|\le C(|\chi_b|+|u_b|),
\end{aligned}
\end{align}
where we used the assumption $|u_b|<1$. Similarly, from the first equation of (\ref{appEout}), we obtain
\begin{align}
\label{appEout02}
\begin{aligned}
r^{n-2}{\eta^E}^{(1)} =& r^{n-2}\dfrac{v_+}{\theta_+}{\chi^E}^{(1)} +r^{n-2}F({\eta^E}^{(0)},{\chi^E}^{(0)})
=&\frac{v_+}{\theta_+}r^{n-2}{\chi^E}^{(1)}+\frac{|u_b|^2v_+}{2R\theta_+r^n}. 
\end{aligned}
\end{align}
It follows from (\ref{appEout01}) and (\ref{appEout02}) that
\begin{align}
\begin{aligned}
|r^{n-2}{\eta^E}^{(1)}(r)| \le C(|\chi_b|+|u_b|).
\end{aligned}
\end{align}
Hence, we choose \(M\) so that 
\begin{align}
\label{M1Eout}
C(|\chi_b|+|u_b|) \le M,
\end{align}
which ensures that $\|{\chi^{E}}^{(1)}\|_X, \ \|{\eta^{E}}^{(1)}\|_X \le M$. 

\smallskip

\noindent
\underline{Case $m=k$}\ $(k\ge 1)$.\quad
Assume that $\|{\eta^{E}}^{(k)}\|_X,\ \|{\chi^{E}}^{(k)}\|_X \le M$. 
We further require
\begin{align}
\label{M2Eout}
M\le \frac{v_+}{2},
\end{align}
which, in particular, implies
\begin{align}
\label{bound}
|{\eta^{E}}^{(k)}(r)|\le \frac{v_+}{2},\quad 
|{\e^{E}}^{(k)}|=|\frac{u_b}{v_++{\eta^{E}}^{(k)}(1)}|\le |\frac{2u_b}{v_+}|,
\quad r \ge 1.
\end{align}
We first estimate \(\chi^{E(k+1)}\) defined by (\ref{appEout}):
\begin{equation}
\begin{aligned}
\left|r^{n-2}\chi^{E(k+1)}(r)\right|
&\le |\chi_b|
   + C|u_b|M
   + C|u_b|^3(v_+ + M)^2 \\
&\le C\left(
|\chi_b|
+ |u_b|
+ |u_b|M
+ |u_b|M^2
\right).
\end{aligned}
\label{estchiEout}
\end{equation}
We next estimate \(\eta^{E(k+1)}\) using (\ref{estchiEout}).
 Multiplying the first equation of (\ref{appEout}) by $r^{n-2}$ and taking the absolute value, we obtain
\begin{align}
&\left|r^{n-2}\eta^{E(k+1)}(r)\right|\\
\leq {}&
\left|
\frac{v_+}{\theta_+}r^{n-2}\chi^{E(k+1)}
-\frac{v_+\bigl(r^{n-2}\chi^{E(k)}\bigr)\eta^{E(k)}}
{\theta_+\bigl(v_+ + \eta^{E(k)}\bigr)}
+\frac{\bigl(r^{n-2}\eta^{E(k)}\bigr)\eta^{E(k)}}
{v_+ + \eta^{E(k)}}
\right|
\nonumber\\
&+\left|
\frac{r^{n-2}\left|\epsilon^{E(k)}\right|^2v_+^3}
{2R\theta_+r^{2(n-1)}}
+\frac{r^{n-2}\left|\epsilon^{E(k)}\right|^2v_+^2\eta^{E(k)}}
{R\theta_+r^{2(n-1)}}-\frac{r^{n-2}\left|\epsilon^{E(k)}\right|^2(n-1)v_+^2}
{R\theta_+}
\int_r^\infty
\frac{\eta^{E(k)}(s)}{s^{2n-1}}\,ds
\right|
\nonumber\\
\leq {}&
C\left(
|\theta_b - \theta_+|
+|u_b|
+|u_b|M
+\bigl(|u_b|+1\bigr)M^2
\right).
\label{estetaEout}
\end{align}
Here, we take $|u_b|$ and $|\chi_b|$ sufficiently small and choose $M$ so that
\begin{equation}
\label{econdi2}
C(|\chi_b|+|u_b|)\le \frac{M}{3}, \qquad C|u_b|\le \frac{1}{3},\qquad C(|u_b|+1)M\le \frac{1}{3}.
\end{equation}
(\ref{estchiEout}) and (\ref{estetaEout})  give the desired estimate $\|{\eta^{E}}^{(k+1)}\|_X, \ \|{\chi^{E}}^{(k+1)}\|_X \le M$. By elementary calculations, it is easy to see that there exists a
positive constant $\delta$ such that if 
$|u_b|+|\chi_b| \le \delta$, the assumptions 
(\ref{M1Eout}), (\ref{M2Eout}) and (\ref{econdi2}) hold, and, in particular, \(M\) can be chosen as
\begin{equation*}
M = 3C(|\chi_b|+|u_b|),
\end{equation*}
which proves the uniform boundedness of 
$({\eta^{E}}^{(m)}, {\chi^{E}}^{(m)})$  $(m\ge 0)$ in $X$ with the estimate (\ref{eu}). 
Once the uniform bound \eqref{eu} has been established, it remains to
prove that the approximate sequence
$\{(\eta^{E(m)},\chi^{E(m)})\}_{m\geq 0}$
is a Cauchy sequence in $X\times X$.
For $m\geq 1$, we set
\[
D_m
:=
\|\eta^{E(m)}-\eta^{E(m-1)}\|_X
+
\|\chi^{E(m)}-\chi^{E(m-1)}\|_X.
\]
From the definition of $\epsilon^{E(m)}$ and \eqref{bound},
we have
\[
\begin{aligned}
 |\epsilon^{E(m)}-\epsilon^{E(m-1)}|
 &=
 |u_b|
 \left|
 \frac{1}{v_++\eta^{E(m)}(1)}
 -
 \frac{1}{v_++\eta^{E(m-1)}(1)}
 \right|  \\
 &\le
 C|u_b|
 \|\eta^{E(m)}-\eta^{E(m-1)}\|_X
 \le C|u_b|D_m .
\end{aligned}
\]
Using this estimate together with the uniform bound \eqref{eu},
and estimating the differences of the nonlinear terms
in the same manner as in the proof of the uniform bound, we obtain
\[
 D_{m+1}
 \le C\bigl(|u_b|+|\chi_b|\bigr)D_m,
 \qquad m\ge1.
\]
Hence, by taking $|u_b|+|\chi_b|$ sufficiently small, we may assume that
\[
C\bigl(|u_b|+|\chi_b|\bigr)\leq \frac12.
\]
It then follows that
\[
D_{m+1}\leq \frac12 D_m,
\qquad m\geq 1.
\]
Consequently,
\[
D_m\leq 2^{-(m-1)}D_1,
\qquad m\geq 1,
\]
and hence
$\{(\eta^{E(m)},\chi^{E(m)})\}_{m\geq 0}$
is a Cauchy sequence in $X\times X$.

Passing to the limit as $m\to\infty$ in the approximate system
\eqref{appEout}, we obtain a solution $(\eta^E,\chi^E)$ satisfying
the desired decay estimate \eqref{decayE1}.
The regularity and uniqueness of the solution follow from standard
arguments.
Finally, the decay estimate for $(\eta_r^E,\chi_r^E)$ in
\eqref{decayE22} follows directly from \eqref{E_r2} and
\eqref{chi^E_r}.  \quad $\Box$\\
\subsection{Zero-viscosity limit for the outflow problem}
In this subsection, we consider the zero-viscosity limit for the outflow problem. First, we rewrite (\ref{reformE}) in the form
\begin{align}
\label{out1}
 \left\{
\begin{aligned}
&-\frac{\big(\tilde p(v_++\eta^{E}, \theta_++\chi^E)- \tilde p(v_+, \theta_+)\big)}{\epsilon^{E} \mu}r^{n-1}
-\frac{{\epsilon^{E}} v_+}{2\mu}\frac{1}{r^{n-1}}\\[2mm]
&\qquad\qquad\qquad
-\frac{{\epsilon^{E}}\eta^{E}}{\mu r^{n-1}}+\frac{{\epsilon^{E}}(n-1)r^{n-1}}{\mu}\int_r^{\infty}\frac{\eta^{E}(s)}{s^{2n-1}}ds=0,\quad r \ge 1,\\[2mm]
&\displaystyle{\lim_{r\to \infty}}\eta^{E}(r) =0, \qquad \displaystyle{\lim_{r\to \infty}}\chi^{E}(r) =0.
\end{aligned}
 \right.
\end{align}
\noindent
We introduce new variables \(V(r)\) and \(\Theta(r)\) by $V(r):=\eta(r)-\eta^{E}(r)$ and $\Theta(r):=\chi(r)-\chi^{E}(r)$.  
Using the definitions of \(\epsilon\) and \(\epsilon^E\), together with the relations 
\begin{align}
\label{eaeE}
\epsilon-\epsilon^E=-\frac{u_bV(1)}{(v_++\eta(1))(v_++\eta^E(1))},\qquad \quad \frac{1}{\epsilon}-\frac{1}{\epsilon^E}=\frac{V(1)}{u_b},
\end{align}
and subtracting (\ref{out1}) from (\ref{reforme}), we obtain 
%
\begin{align}
\label{out2}
 \left\{
\begin{aligned}
&V_r-\frac{r^{n-1}}{\epsilon\mu}a(r)V= -\eta^{E}_r+F[V, \Theta](r),\quad r \ge 1,\\[2mm]
&\displaystyle{\lim_{r\to \infty}}V(r) =0,
\end{aligned}
 \right.
\end{align}
where
$$
a(r) := \frac{\tilde p(v_++\eta^{E}+V, \theta_++\chi^E+\Theta)-\tilde p(v_++\eta^{E}, \theta_++\chi^E+\Theta)}{V},
$$
and
\begin{align*}
\begin{aligned}
F[V, \Theta](r) :=&\frac{r^{n-1}}{\epsilon\mu}\left(\tilde p(v_++\eta^{E}, \theta_++\chi^E+\Theta)-\tilde p(v_++\eta^{E}, \theta_++\chi^E)\right) \\
&+\frac{r^{n-1}V(1)}{\mu u_b}\left(\tilde p(v_++\eta^{E}, \theta_++\chi^E)-\tilde p(v_+, \theta_+)\right)-\frac{V(1)u_b}{(\eta^{E}(1)+v_+)(\eta(1)+v_+)}\frac{v_+}{2\mu r^{n-1}}\\[2mm]
&+\frac{u_bV(r)}{\mu r^{n-1}(\eta(1)+v_+)}-\frac{V(1)u_b}{(\eta^{E}(1)+v_+)(\eta(1)+v_+)}
\frac{\eta^{E}(r)}{\mu r^{n-1}}\\[2mm]
&-\frac{u_b(n-1)r^{n-1}}{\mu(\eta(1)+v_+)}\int_r^{\infty}\frac{V(s)}{s^{2n-1}}ds
+\frac{V(1)u_b(n-1)r^{n-1}}{\mu (\eta^{E}(1)+v_+)(\eta(1)+v_+)}
\int_r^{\infty}\frac{\eta^{E}(s)}{s^{2n-1}}ds\\[2mm]
&=: I_1+I_2+I_3+I_4+I_5+I_6+I_7.
\end{aligned}
\end{align*}
Here we note that we already know the existence of the solution 
$\eta(r)$ of (\ref{reforme}) and $\eta^{E}(r)$ of (\ref{reformE}) with the estimates (\ref{decayo}) and (\ref{decayE1}) respectively. 
To estimate \(V\) in (\ref{out2}), we use the following lemma.
\begin{lem}(see \cite{HM0, H-M1})
\label{hyokaG}
Define \(G(r,s)\) by
\begin{align*}
G(r,s) := e^{-\frac{\omega}{\mu|\epsilon|}(r^n-s^{n})},\qquad \omega:=\frac{R\theta_+}{v_+^2n}.
\end{align*}
Then, the following statements hold:\\
{\rm (I)} If $|\epsilon|\le \frac{n\omega}{(n-2)\mu}$, then   
\begin{align}
\label{hyoka1}
r^{n-2}G(r,1)=r^{n-2}e^{-\frac{\omega}{\mu|\epsilon|}(r^n-1)}\le 1, \qquad r\ge 1.
\end{align}
{\rm (II)}
If  $l\ge -n$, 
then there exists a positive constant \(C_0\), independent of \(\epsilon\) and \(\mu\), such that, for any \(f\in X_l\),
\begin{align}
\label{ineqG}
\begin{aligned}
&r^{l+n-1}|\int_1^r G(r,s) f(s) ds|
\le C_0|\epsilon| \mu\|f\|_{X_l},\quad r\ge 1.
\end{aligned}
\end{align}
\end{lem}
\noindent
We now estimate \(V\) in (\ref{out2}) using Lemma \ref{hyokaG}. Applying Duhamel's principle to (\ref{out2}), we obtain
\begin{align}
\label{intchi}
\begin{aligned}
V(r)&=-\int_r^{\infty}e^{-\int_{r}^{s}\frac{a(\tau)}{\epsilon\mu}\tau^{n-1}d\tau}(-\eta^{E}_r+F[V,\Theta])(s)\,ds\\
&=-\int_r^{\infty}e^{-\int_{r}^{s}\frac{a(\tau)}{\epsilon\mu}\tau^{n-1}d\tau}
(-\eta^{E}_r+I_1+I_2+I_3+I_4+I_5+I_6+I_7)(s)\,ds.
\end{aligned}
\end{align}
By the assumption $\tilde p_v<0\ (v>0)$, Theorem \ref{rmt}, and (\ref{decayE1}), 
it is easy to see that if $|u_b|+|\chi_b|$ is suitably small, there exists a positive constant 
$\delta$, independent of \(\mu\), such that
\begin{equation*}
-a(r) \ge \delta, \qquad r\ge 1.
\end{equation*}
Therefore, using the estimates for \(\eta^E\) and \(\chi^E\) obtained in Theorem \ref{rmtE}, we estimate the first, second, and third terms on the right-hand side of (\ref{intchi}) as follows: 
\begin{align}
\label{chiu1}
\begin{aligned}
&|\int_r^{\infty}e^{-\int_{r}^{s}\frac{a(\tau)}{\epsilon\mu}\tau^{n-1}d\tau}\eta^{E}_r(s)\, ds| \le  \int_r^{\infty}e^{-\frac{v(1)\delta}{|u_b|\mu}(s-r)}|\eta^{E}_r(s)|\, ds \\
&\qquad \quad \le \int_r^{\infty}e^{-\frac{v(1)\delta}{|u_b|\mu}(s-r)}ds 
\cdot \sup_{s\ge r}|\eta^{E}_r(s)| 
\le \frac{|u_b|\mu}{v(1)\delta}\cdot \sup_{s\ge r}|\eta^{E}_r(s)|
\le \ C \mu/r^{n-1}, \\
&|\int_r^{\infty}e^{-\int_{r}^{s}\frac{a(\tau)}{\epsilon\mu}\tau^{n-1}d\tau}I_1(s)\,ds| \\
& \le  \frac{|v(1)|}{\mu |u_b|}\int_r^{\infty}e^{-\frac{v(1)\delta}{|u_b|\mu n}(s^n-r^n)}s^{n-1}|\left(\tilde p(v_++\eta^{E}, \theta_++\chi^E+\Theta)-\tilde p(v_++\eta^{E}, \theta_++\chi^E)\right)| \,ds \\
&\quad \le  \frac{C|v(1)|}{\mu |u_b|}\int_r^{\infty}e^{-\frac{v(1)\delta}{|u_b|\mu n}(s^n-r^n)} s^{n-1} \  ds \cdot \sup_{r\ge 1}|r^{n-2}\Theta(r)|/r^{n-2} \ \le C\sup_{r\ge 1}|r^{n-2}\Theta(r)|/r^{n-2},\\
&|\int_r^{\infty}e^{-\int_{r}^{s}\frac{a(\tau)}{\epsilon\mu}\tau^{n-1}d\tau}I_2(s)\,ds| \\
& \le  \frac{|V(1)|}{\mu |u_b|}\int_r^{\infty}e^{-\frac{v(1)\delta}{|u_b|\mu n}(s^n-r^n)}s^{n-1}|\left(\tilde p(v_++\eta^{E}, \theta_++\chi^E)-\tilde p(v_+, \theta_+)\right)| \,ds \\
&\quad \le  \frac{C\sup_{r\ge 1}|r^{n-2}V(r)|}{\mu |u_b|}\int_r^{\infty}e^{-\frac{v(1)\delta}{|u_b|\mu n}(s^n-r^n)} s^{n-1}ds \cdot \sup_{r\ge 1}r^{n-2}|(|\eta^E(r)|+|\chi^E(r)|)/r^{n-2} \\
&\quad \le C(|u_b|+|\chi_b|)\sup_{r\ge 1}|r^{n-2}V(r)|/r^{n-2},
\end{aligned}
\end{align}
By performing the same calculations on the remaining terms on the right-hand side of (\ref{intchi}), we can obtain
\begin{align}
\label{chiu3}
\begin{aligned}
&|\int_r^{\infty}e^{-\int_{r}^{s}\frac{a(\tau)}{\epsilon\mu}\tau^{n-1}d\tau}
(I_3+I_4+I_5+I_6+I_7)(s)\, ds| \le C|u_b|\sup_{r\ge 1}|V(r)|/r^{n-2}.
\end{aligned}
\end{align}
Substituting (\ref{chiu1}) and (\ref{chiu3}) into (\ref{intchi}), we obtain the estimate 
\begin{align}
\label{chiu4}
\sup_{r\ge 1}|r^{n-2}V(r)|\le C\mu+C\sup_{r\ge 1}|r^{n-2}\Theta(r)|+C(|u_b|+|\chi_b|)\sup_{r\ge 1}|r^{n-2}V(r)|.
\end{align}
Next, we estimate \(|\Theta(r)|\).  
Integrating the second equation of (\ref{reforme}) over \([r,\infty)\), we obtain
\begin{align}
\label{chiintout}
 \left\{
\begin{aligned}
&\chi(r) = \dfrac{\alpha}{\kappa(n-2)r^{n-2}}+H(\eta,\eta_r,\chi)(r),
\\
&\qquad \alpha = \kappa(n-2)\chi_b-\kappa(n-2)
H(\eta,\eta_r,\chi)(1),\\[10pt]
&\qquad \quad H(\eta,\eta_r,\chi)(r):= -\int_r^\infty \dfrac{1}{\kappa \tau^{n-1}}
\Big\{{\epsilon c_P}\chi + 
\frac{\epsilon^3}{2}\frac{(v_++\eta)^2}{\tau^{2(n-1)}}\\
&\hspace{5cm} -\epsilon^2\mu \dfrac{(v_++\eta)\eta_r}{\tau^{(n-1)}}
+\frac{2\e^2\nu(n-1)(v_++\eta)^2}{\tau^n}\Big\}\,d\tau.
\end{aligned}
\right.
\end{align}
Subtracting (\ref{chi^Eintout}) from (\ref{chiintout}), we obtain\\
\begin{align}
\label{ThetaE1}
\begin{aligned}
\Theta(r)=-\dfrac{H(\eta,\eta_r,\chi)(1)-H^E(\eta^E,\chi^E)(1)}{r^{n-2}}+H(\eta,\eta_r,\chi)(r)-H^E(\eta^E,\chi^E)(r),
\end{aligned}
\end{align}
where
\begin{align}
\label{H-H^E}
\begin{aligned}
&H(\eta,\eta_r,\chi)(r)-H^E(\eta^E,\chi^E)(r)\\
=&-\int_r^\infty \dfrac{1}{\kappa \tau^{n-1}}
\Big\{c_P(\epsilon\chi-\epsilon^E\chi^E) + 
\frac{\epsilon^3(v_++\eta)^2-(\epsilon^E)^3(v_++\eta^E)^2}{2\tau^{2(n-1)}}\\
&\hspace{4cm}-\epsilon^2\mu \dfrac{(v_++\eta)\eta_r}{\tau^{(n-1)}}
+\frac{2\e^2\nu(n-1)(v_++\eta)^2}{\tau^n}\Big\}\,d\tau\\
=:&J_1+J_2+J_3+J_4.
\end{aligned}
\end{align}
Noting (\ref{eaeE}), we have 
\begin{align}
\label{J1o}
\begin{aligned}
|J_1|&=\left|\frac{c_P}{\kappa}\int_r^\infty \dfrac{\epsilon\chi-\e^E\chi^E}{\tau^{n-1}}d\tau\right|=\left|\frac{c_P}{\kappa}\int_r^\infty \dfrac{\epsilon\Theta+(\e-\e^E)\chi^E}{\tau^{n-1}}d\tau\right|\\
&\le C|u_b|\left(\sup_{r\ge 1}|r^{n-2}\Theta(r)|+\sup_{r\ge 1}|r^{n-2}V(r)|\right)/r^{2(n-2)},
\end{aligned}
\end{align}
\begin{align}
\label{J2o}
|J_2|&=\left|\frac{1}{2\kappa}\int_r^\infty \dfrac{\epsilon^3(v_++\eta)^2-(\epsilon^E)^3(v_++\eta^E)^2}{ \tau^{3(n-1)}}d\tau\right|\le C|u_b|^3\sup_{r\ge 1}|r^{n-2}V(r)|/r^{3n-4}. 
\end{align}
By the definition of $F$ in (\ref{out2}), we obtain that 
\begin{align}
|F[V, \Theta](r)|\le \frac{C r^{n-1}}{|\e|\mu}\left(\sup_{r\ge 1}|r^{n-2}\Theta|+\sup_{r\ge 1}|r^{n-2}V(r)|\right),
\end{align}
which in turn yields 
\begin{align}
\label{V_rO}
|V_r|\le \frac{C r^{n-1}}{|\e|\mu}\left(\sup_{r\ge 1}|r^{n-2}\Theta|+\sup_{r\ge 1}|r^{n-2}V(r)|\right)+\frac{C(|u_b|+|\chi_b|)}{r^{n-1}}.
\end{align}
By combining (\ref{decayE22}) and (\ref{V_rO}), we obtain that 
\begin{align}
\label{J3o}
\begin{aligned}
|J_3|&=\left|\frac{\e^2\mu}{\kappa}\int_r^\infty \dfrac{(v_++\eta)\eta_r}{\tau^{2(n-1)}}d\tau\right|\le \left|\frac{\e^2\mu}{\kappa}\int_r^\infty \dfrac{(v_++\eta)(V_r+\eta^E_r)}{\tau^{2(n-1)}}d\tau\right|\\
&\le C|u_b|\left(\sup_{r\ge 1}|r^{n-2}\Theta|+\sup_{r\ge 1}|r^{n-2}V(r)|+\mu\right)/r^{n-2}, 
\end{aligned}
\end{align}
\begin{align}
\label{J4o}
|J_4|= \left|\frac{2\e^2\nu(n-1)}{\kappa}\int_r^\infty \frac{(v_++\eta)^2}{\tau^{2n-1}}d\tau\right|\le C\mu/r^{2(n-1)}.
\end{align}
Substituting (\ref{J1o}), (\ref{J2o}), (\ref{J3o}) and (\ref{J4o}) into (\ref{H-H^E}), we obtain  
\begin{align}
\begin{aligned}
&|H(\eta,\eta_r,\chi)(r)-H^E(\eta^E,\chi^E)(r)|\\
&\le C|u_b|\left(\sup_{r\ge 1}|r^{n-2}\Theta|+\sup_{r\ge 1}|r^{n-2}V(r)|+\mu\right)/r^{n-2}+C\mu/r^{n-2}.
\end{aligned}
\end{align}
Applying the above inequality to (\ref{ThetaE1}), we obtain that 
\begin{align}
\label{ThetaE1e}
\begin{aligned}
\sup_{r\ge 1}|r^{n-2}\Theta(r)|\le C|u_b|\left(\sup_{r\ge 1}|r^{n-2}\Theta|+\sup_{r\ge 1}|r^{n-2}V(r)|\right)+C\mu.
\end{aligned}
\end{align}
Finally, combining (\ref{chiu4}) and (\ref{ThetaE1e}), we obtain the following estimate.
\begin{align}
\label{outfinal}
\sup_{r\ge 1}|r^{n-2}V(r)|+\sup_{r\ge 1}|r^{n-2}\Theta(r)|\le C\mu+C(|u_b|+|\chi_b|)\left(\sup_{r\ge 1}|r^{n-2}\Theta|+\sup_{r\ge 1}|r^{n-2}V(r)|\right).
\end{align}
Hence, taking $|u_b|+|\chi_b|$ suitably small, we obtain
\begin{align}
\label{chiu5}
\sup_{r\ge 1}|V(r)|+\sup_{r\ge 1}|\Theta(r)|=
\sup_{r\ge 1}|(\eta-\eta^{E})(r)|+\sup_{r\ge 1}|(\chi-\chi^{E})(r)|
\le C\mu. 
\end{align}
Since
\[
 \rho=\frac{1}{v_++\eta}, \qquad
 \rho^E=\frac{1}{v_++\eta^E},
 \qquad
 \theta-\theta^E=\Theta,
\]
it follows from \eqref{chiu5} that
\[
 \sup_{r\ge1}|\rho(r)-\rho^E(r)|
 +\sup_{r\ge1}|\theta(r)-\theta^E(r)|
 \le C\mu.
\]
Finally, recalling that $u$ and $u^{E}$ are respectively given by
$$
u(r)=\frac{u_b(v_++\eta(r))}{r^{n-1}(v_++\eta(1))},\quad
u^{E}(r)=\frac{u_b(v_++\eta^{E}(r))}{r^{n-1}(v_++\eta^{E}(1))},\quad r\ge 1,
$$
we can also show
$$
\sup_{r\ge 1}|(u-u^{E})(r)| \le C\mu. 
$$
This completes the proof of statement (II) in Theorem \ref{etE}. \qquad $\Box$


\section{Inflow problem}
In this section, 
we consider the case \(u_b>0\) and prove the statement in (III) of Theorem \ref{bl+E}. 
That is, we show that the inviscid limit of the solution of the stationary problem (\ref{nrs0}) is described by a superposition of the solution of the inviscid problem (\ref{Es})  with boundary conditions \(\rho^E(r_0)u^E(r_0)=\rho_bu_b\) and \(\theta^E(r_0)=\theta_b\)  and the boundary-layer solution  of (\ref{Ebl}).

\subsection{Inviscid problem} 
\label{isi}
In this subsection, we only make a brief comment on the existence of a solution to the inviscid problem (\ref{Es}) corresponding to the inflow problem, since the argument is almost the same as that in Subsection 3.1.

In the inflow case, we seek a solution \((\rho^E,u^E,\theta^E)\) of (\ref{Es}) satisfying
\begin{align}
\label{4n1e}
r^{n-1}\rho^{E}(r) u^{E}(r)=\epsilon := \rho_b u_b, \qquad r\ge 1.
\end{align}
To construct the solution, we introduce \(\eta^E(r):=v^E(r)-v_+\). By the same argument as in the derivation of (\ref{uE})--(\ref{E2}), we consider the following integral equation for $\eta^E$:
\begin{align}
\label{4E2}
\begin{aligned}
\eta^{E}(r)&=\int_r^{\infty} \frac{s^{-(2n-1)}\epsilon^2(n-1)(v_++\eta^{E})-\ti{p}_{\theta}(v_++\eta^E, \theta_++\chi^E)\chi^E_s}{|\tilde p_v(v_++\eta^{E}, \theta_++\chi^E)|-\epsilon^2s^{-2(n-1)}} ds.
\end{aligned}
\end{align}
Moreover, proceeding as in the derivation of (\ref{ieEout}) and (\ref{chi^Eintout}), we see that the system we need to solve is equivalent to
\begin{align}
\label{4reformE}
 \left\{
\begin{aligned}
&\widetilde{p}(v_+ + \eta^E,\theta_+ + \chi^E)
-\widetilde{p}(v_+,\theta_+)\\
&\quad +\frac{\epsilon^2 v_+}{2r^{2(n-1)}}
+\frac{\epsilon^2\eta^E}{r^{2(n-1)}}
-\epsilon^2(n-1)
\int_r^\infty
\frac{\eta^E(s)}{s^{2n-1}}\,ds
=0,
\qquad r>1,\\[2mm]
&\chi^E_r=-\frac{\alpha^E}{\kappa r^{n-1}}+ H^E_0(\eta^E,\chi^E)(r), \quad r>1,\\[2mm]
&H^E_0(\eta^E,\chi^E)(r)=
\frac{\epsilon c_P}{\kappa r^{n-1}}\chi^E + 
\frac{\epsilon^3}{2}\frac{(v_++\eta^E)^2}{\kappa r^{3(n-1)}},\\
&\displaystyle{\lim_{r\to \infty}}\eta^E(r) =0,\quad \displaystyle{\lim_{r\to \infty}}\chi^E(r) =0, \quad \chi^E(1)= \chi_b,
\end{aligned}
 \right.
\end{align}
where we define
$H^E(\eta^E,\chi^E)(r):= -\int_r^\infty H_0^E(\eta^E,\chi^E)(\tau)\ d\tau$
and $\alpha^E$ is defined as
\begin{align}
\label{alphaE}
\alpha^E := \kappa(n-2)\chi_b-\kappa(n-2)H^E(\eta^E,\chi^E)(1),
\end{align}
so that $\chi^E(1)=\chi_b$. \\
By the same argument as in the proof of Theorem \ref{rmtE}, we obtain the following result for the reformulated problem (\ref{4E2}) and (\ref{4reformE}). 
\begin{thm} 
\label{rmtEI}
Let $n \ge 3$. Then, for any $v_+>0$, and  $\theta_+ >0$, there exist 
positive constants $\epsilon_0$ and $C$ such that if $|u_b|+|\rho_b-\rho_+|+|\chi_b| \le \epsilon_0$, 
there exists a unique smooth solution $(\eta^{E}, \chi^E)$ of the problem (\ref{4reformE}) satisfying
\begin{subequations}
\begin{align}
\label{4decayE1}
&|\eta^{E}(r)|,\ |\chi^{E}(r)|\le Cr^{-(n-2)}(|u_b|+|\chi_b|), \qquad r\ge 1, \\[2mm]
\label{decayE2}
&|\eta^{E}_r(r)|,\ |\chi^{E}_r(r)| \le Cr^{-(n-1)}(|u_b|+|\chi_b|), \ \ \qquad r\ge 1.
\end{align}
\end{subequations}
\end{thm}
The proof is essentially the same as that of Theorem~\ref{rmtE}.
Indeed, the only difference is that, in the present inflow problem,
$\epsilon=\rho_bu_b$ is a fixed constant, whereas in the outflow problem
$\epsilon^{E}$ depends on the boundary value $\eta^E(1)$.
Thus, repeating the same iteration and contraction argument as in the
proof of Theorem~\ref{rmtE}, we obtain the desired estimates. 
We therefore omit the details. 

Once Theorem \ref{rmtEI} is proved, statement (I) in Theorem \ref{bl+E} follows from (\ref{4n1e}) together with the definitions of \(\rho^E\) and \(\theta^E\).

\subsection{Boundary layer solution} 
We derive the boundary layer equation and construct its solution by introducing a new independent variable \(y\) defined by
\begin{align}
\label{bltrans}
y:=\frac{r-1}{\mu}.
\end{align}
By applying the transformation \eqref{bltrans} to equations
\eqref{Narrange} and \eqref{enea} and formally letting $\mu\to0$
in the resulting equations, we obtain the following boundary layer equations:
\begin{align}
\label{Narrange2}
 \left\{
\begin{aligned}
&\epsilon \eta^B_{yy}=\tilde p(v_++\eta^B, \theta_++\chi^B)_y+\epsilon^2 \eta^B_y,\qquad y\ge 0,\\
&\chi^B_{yy}=0,\qquad y\ge 0.
\end{aligned}
\right.
\end{align}
We seek a solution \((\eta^B,\chi^B)\) satisfying the boundary and far-field conditions. Integrating (\ref{Narrange2}) over $[y, \infty)$ 
under the conditions $\displaystyle{\lim_{y\to \infty}}\eta^B(y) =\eta^{E}(1)$ , $\displaystyle{\lim_{y\to \infty}} \eta^B_y =0$ and $\displaystyle{\lim_{y\to \infty}}\chi^B(y) =\chi^{E}(1)=\chi_b$,  
where $\eta^E(1)$ and $\chi^E(1)$ are the boundary values of the
inviscid solution obtained in Subsection \ref{isi}, we obtain
\begin{align}
\label{Narrange3}
 \left\{
\begin{aligned}
  &\eta^B_y= 
\frac{1}{\epsilon}\left(\tilde p(v_++\eta^B(y), \theta_++\chi^B(y))- \tilde p(v_++\eta^{E}(1), \theta_++\chi^E(1))\right)\\
&\hspace{6cm}+\e (\eta^B(y)-\eta^{E}(1)), \quad {y \ge 0}, \\
&\chi^B(y)\equiv \chi_b, \quad {y \ge 0},  \\[2mm]
 &\displaystyle{\lim_{y\to \infty}}\eta^B(y) =\eta^{E}(1), 
 \quad \eta^B(0)= \eta_b:=v_b-v_+,
 \end{aligned}
 \right.
\end{align}
where $\e=u_b/v_b$. 
We call (\ref{Narrange3}) the boundary-layer equation and its solution the boundary-layer solution. 
The following theorem gives the existence and uniqueness of the boundary layer solution to (\ref{Narrange3}). 
\begin{thm} 
\label{blexistB}
Let $n \ge 3$. Then, for any $v_+>0$ and $\theta_+>0$, there exist positive constants 
$\epsilon_0$, $\nu_0$ and $C$ such that if $|u_b|+|\eta_b-\eta^{E}(1)|+|\chi_b| \le \epsilon_0$, 
there exists a unique smooth solution \((\eta^B,\chi^B)\) of the problem {\rm (\ref{Narrange3})} satisfying
\begin{align}
\label{4decaybl}
|\eta^B(y)-\eta^{E}(1)|\le C|\eta_b-\eta^{E}(1)|
\ e^{-\frac{\nu_0}{\e}y}, \qquad \chi^B(y)= \chi_b, \quad {y \ge 0}, 
\end{align}
\end{thm}
Once Theorem \ref{blexistB} is proved, since \(\rho^B\), \(u^B\), and \(\theta^B\) defined by 
$$
{\rho}^B(y) := \frac{1}{v_++\eta^B(y)},\quad
{u}^B(y) := \frac{u_b}{v_b}(v_++\eta^B(y)), \quad \theta^B(y):=\theta_b 
$$
 give a solution of (\ref{Ebl}), statement (II) of Theorem \ref{bl+E} follows.

\medskip
\noindent
{\it Proof of Theorem \ref{blexistB}.}\qquad We first rewrite the first equation of (\ref{Narrange3}) as 
\begin{align}
\label{Narrange4}
\begin{aligned}
  &(\eta^B-\eta^{E}(1))_y-\frac{\tilde p_v(v_*,\theta_b)}{\e}(\eta^B-\eta^{E}(1))\\
&=\frac{1}{\epsilon}\left\{\tilde p(v_++\eta^B(y), \theta_b)- \tilde p(v_*, \theta_b)-\tilde p_v(v_*, \theta_b)(\eta^B-\eta^{E}(1))\right\}+\e (\eta^B(y)-\eta^{E}(1)),
 \end{aligned}
\end{align}
where $v_*=v_++\eta^{E}(1)$. 
We introduce a new unknown function \(\hat\eta\) as the deviation of \(\eta^B\) from the far-field state \(\eta^E(1)\) by
\begin{equation}
\label{bareta}
\hat \eta(y) := \eta^B(y) - \eta^{E}(1),\qquad y \ge 0.
\end{equation}
Substituting (\ref{bareta}) into (\ref{Narrange4}), we obtain
\begin{align}
\label{Narrange5}
 \left\{
\begin{aligned}
 &\hat \eta_y+a_\e \hat \eta=\frac{1}{\e}N[\hat{\eta}], \quad y\ge 0,\\[2mm]
 &\displaystyle{\lim_{y\to \infty}}\hat \eta(y) =0, \quad \hat \eta(0)= \eta_b - \eta^{E}(1),
 \end{aligned}
 \right.
\end{align}
where 
\begin{align*}
\begin{aligned}
a_\e := -\frac{\tilde p_v(v_*, \theta_b)}{\e}-\e,\qquad N[\hat{\eta}]:=\tilde p(v_*+\hat{\eta}, \theta_b)- \tilde p(v_*, \theta_b)-\tilde p_v(v_*, \theta_b)\hat{\eta}.
\end{aligned}
\end{align*}
Applying Duhamel's principle to (\ref{Narrange5}), we obtain
\begin{align}
\label{Narrange6}
\hat \eta(y)=e^{-a_{\e}y}\hat \eta(0)
+\frac{1}{\e}\int_0^y e^{-a_{\e}(y-s)}N[\hat \eta](s)\,ds.
\end{align}
\noindent
Since $\widetilde p_v(v_+,\theta_+)<0$ and
\[
|\eta^E(1)|\le C(|u_b|+|\chi_b|),
\]
we see that $v_*=v_++\eta^E(1)$ and $\theta_b=\theta_++\chi_b$
remain sufficiently close to $v_+$ and $\theta_+$, respectively,
provided that the boundary data are sufficiently small.
Hence, by the continuity of $\widetilde p_v$, there exist positive
constants $\epsilon_0$ and $\nu_0$ such that, if
\[
|u_b|+|\eta_b-\eta^E(1)|+|\chi_b|\le\epsilon_0,
\]
then
\[
a_\epsilon\ge \frac{\nu_0}{\epsilon}.
\]
We therefore seek a solution of \eqref{Narrange6} in the Banach space
$Y$, equipped with the norm $\|\cdot\|_Y$ defined by
\begin{align*}
Y:=\{g \in C([0,\infty));\ \sup_{y\ge 0}|e^{\frac{\nu_0}{\e}y} g(y)|< \infty \},
\quad \|g\|_Y = \sup_{y\ge 0}|e^{\frac{\nu_0}{\e}y}g(y)|.
\end{align*}
Since $N[\hat{\eta}]=O(\hat{\eta}^2)$, the remaining argument for
constructing the desired solution of \eqref{Narrange6} in $Y$ for
sufficiently small $\hat{\eta}(0)$ is the same as in our previous
paper~\cite{H-M}.
We refer the reader to \cite{H-M} for the details. $\Box$ 

We now change the variable from $y$ back to $r$
for the boundary-layer solution $(\eta^B(y),\chi^B(y))$ and define
$$
{\bar \eta}(r) = {\eta}^B(\frac{r-1}{\mu}) - \eta^{E}(1),\quad \bar{\chi}(r):=\chi^B(\frac{r-1}{\mu})-\chi^E(1).
$$
The boundary layer equation  (\ref{Narrange3}) is rewritten in terms of ${\bar \eta}$ in the form
\begin{align}
\label{reformBr}
 \left\{
\begin{aligned}
  &\bar{\eta}_r- 
\frac{\left(\ti{p}(\bar{\eta}(r)+\eta^E(1)+v_+,\theta_b)-\ti{p}(\eta^E(1)+v_+,\theta_b)\right)}{\e\mu}-\frac{\e}{\mu}\bar{\eta}=0, \quad {r \ge 1}, \\[2mm]
 &\bar{\chi}(r)\equiv 0, \\
 &\displaystyle{\lim_{r\to \infty}}\bar{\eta}(r) =0, 
 \quad \bar{\eta}(1)= \eta_b-\eta^{E}(1),
 \end{aligned}
 \right.
\end{align}
where we recall that $\epsilon = u_b/v_b>0$.  

Theorem \ref{blexistB} can be equivalently restated in terms of \(r\) as follows. 
\begin{cor} 
\label{blexist}
Let $n \ge 3$. Then, for any $v_+>0$ and $\theta_+>0$, there exist positive constants 
$\epsilon_0$, $\nu_0$, and $C$ such that if $|u_b|+|\eta_b-\eta^{E}(1)|+|\chi_b| \le \epsilon_0$, 
there exists a unique smooth solution $\bar{\eta}$ of the problem {\rm (\ref{reformBr})} satisfying
\begin{align}
\label{4decaybl}
|\bar{\eta}|\le C|\eta_b-\eta^{E}(1)|
\ e^{-\nu_0\frac{r-1}{\e\mu}}, \qquad r\ge 1. 
\end{align}
\end{cor}

\noindent
Combining (\ref{reformBr}) and (\ref{4decaybl}), we obtain
\begin{align}
\label{decayetar}
|\bar{\eta}_r| \le C|\eta_b-\eta^{E}(1)|\frac{e^{-\nu_0\frac{r-1}{\e\mu}}}{\e\mu}.
\end{align}
The estimates \eqref{4decaybl} and \eqref{decayetar}
will be used in the proof of the zero-viscosity limit
in the next subsection.

\subsection{Zero-viscosity limit for the inflow problem }
In this subsection, we prove statement (III) of Theorem \ref{bl+E}. 
To prove the assertion for $\rho$ in statement (III) of
Theorem~\ref{bl+E}, it suffices to show that the solution $\eta$ of
\eqref{eta_r} 
converges uniformly to the sum of the solutions of (\ref{4reformE}) and (\ref{reformBr}), namely, \(\eta^E+\bar\eta\), as
$\mu \to 0$. 
To this end, we use the fact that, for any $\alpha\in(0,1)$,
\[
|\bar{\eta}(r)|
\le
C\exp\left\{-\frac{\nu_0\mu^{\alpha-1}}{\epsilon}\right\},
\qquad
r\ge 1+\mu^\alpha,
\]
which follows from \eqref{4decaybl}.
We divide the interval $[1,\infty)$ into
\[
I_b:=[1,1+\mu^\alpha]
\quad\text{and}\quad
I_\infty:=[1+\mu^\alpha,\infty).
\]
We first show that $\eta$ converges uniformly to $\eta^E$
on $I_\infty$, and then to $\eta^E+\bar{\eta}$ on $I_b$,
as $\mu\to0$.

\medskip

\noindent
\underline{\bf Zero-viscosity limit on $I_{\infty}$}

\smallskip

We define new variables $V$ and $\Theta$ by
\[
V(r):=\eta(r)-\eta^E(r),\qquad
\Theta(r):=\chi(r)-\chi^E(r).
\]
Subtracting \eqref{4reformE} from \eqref{reforme}, we obtain
the following system for $(V,\Theta)$:
\begin{align}
\label{inf0}
 \left\{
\begin{aligned}
&V_r+\frac{r^{n-1}}{\epsilon\mu}a(r)V= 
-\eta^{E}_r+F,\quad r > 1,\\[2mm]
&\Theta_r-\frac{\e c_P}{\kappa r^{n-1}}\Theta=N[V,\Theta](r),\quad r > 1,\\[2mm]
&\displaystyle{\lim_{r\to \infty}}V(r) =0,\ V(1) = \eta_b-\eta^{E}(1),\ \displaystyle{\lim_{r\to \infty}}\Theta(r) =0,\ \Theta(1) = 0,\\[2mm]
&a := -\frac{\tilde p(v_++\eta^{E}+V,\chi^E+\theta_++\Theta)-\tilde p(v_++\eta^{E},\chi^E+\theta_++\Theta)}{V}-\frac{\e^2}{r^{2(n-1)}},\\ 
&F[V,\Theta]:=-\frac{\e (n-1)r^{n-1}}{\mu}\int_r^{\infty}\frac{V(s)}{s^{2n-1}}ds\\
&\hspace{3cm}+\frac{r^{n-1}}{\epsilon\mu}(\ti{p}(\eta^E+v_+,\chi^E+\theta_++\Theta)-\ti{p}(\eta^E+v_+,\chi^E+\theta_+)),\\
&N[V,\Theta](r):=-\frac{\al -\al^E}{\kappa r^{n-1}}+\frac{\e^3}{2\kappa r^{3(n-1)}}(2v_++V+2\eta^E)V\\
&\hspace{3cm} -\e^2\mu\frac{(v_++V+\eta^E)(V+\eta^E)_r}{\kappa r^{2(n-1)}}+
2\e^2\nu(n-1)\frac{(v_++V+\eta^E)^2}{\kappa r^{2n-1}}.
\end{aligned}
 \right. 
\end{align}
Using the assumption \(\widetilde p_v(v,\theta)<0\) for \(v>0\) together with 
Theorem \ref{rmtEI}, 
we see that, for suitably small \(|u_b|+|v_b-v_+|+|\theta_b-\theta_+|\), there exists a positive constant \(\delta\), independent of \(\mu\), such that
\[
a(r)\ge\delta,\qquad r\ge1.
\] 
We have the following theorem for the reformulated problem (\ref{inf0}). 
\begin{thm} 
\label{ivli0}
Let $n \ge 3$ and $\al \in (0,1)$. Then, for any $v_+>0$ and $\theta_+>0$, there exist positive constants $\epsilon_0$ 
and $C$, independent of $\mu\in (0,1]$, 
such that if $|u_b|+|v_b-v_+|+|\theta_b-\theta_+| \le \epsilon_0$, 
the solution \((V,\Theta)\) of (\ref{inf0}) satisfies
\begin{align}
\label{invsidlimitinf1}
\begin{aligned}
\sup_{r\ge 1+\mu^{\alpha}}|V(r)|+\sup_{r\ge 1+\mu^{\alpha}}|\Theta(r)| \le  C\mu^{\al}.
\end{aligned}
\end{align}
\end{thm}

\noindent
To prove the above theorem, we define the Banach space $X_l$ as 
\[
X_l:=\left\{f\in C([1,\infty))\ ;\ \sup_{r\ge1}|r^l f(r)|<\infty\right\},
\qquad
\|f\|_{X_l}:=\sup_{r\ge1}|r^l f(r)|.
\]
We first establish the following lemmas.
\begin{lem}
\label{hyoka1}
{\rm (I)} (see \cite{HM0}) If $\epsilon\le \frac{\delta}{(n-2)\mu}$, then   
\begin{align}
\label{hyoka}
r^{n-2}G(r,1)\le 1, \qquad r\ge 1.
\end{align}
{\rm (II)}
If  $l\ge -n$, then 
there exists a positive constant \(C_0\), independent of \(\epsilon\) and \(\mu\), such that, for any \(f\in X_l\),
\begin{align}
\label{ineqG}
\begin{aligned}
&r^{l+n-1}|\int_1^r G(r,s) f(s) ds|
\le C_0\epsilon \mu\|f\|_{X_l},\quad r\ge 1,
\end{aligned}
\end{align}
where $G(r,s):=
\exp\left\{
-\int_s^r
\frac{|a(\tau)|}{\epsilon\mu}\tau^{n-1}\,d\tau
\right\}.$
\end{lem}
\noindent
We also establish the following lemma. 
\begin{lem}
\label{estaV}
For $V$, $\Theta$, and $\alpha-\alpha^E$ appearing in \eqref{inf0},
we have the following estimates.
\begin{align}
\label{nlV}
\begin{aligned}
&{\rm (i)}\quad |V_r|\le C(\e+|\chi_b|)+C\frac{r^{n-1}}{\e\mu}(\sup_{r\ge 1+\mu^{\alpha}}|V(r)|+\sup_{r\ge 1+\mu^{\alpha}}|\Theta(r)|),\qquad r\ge 1+\mu^{\al},\\
&{\rm (ii)}\quad |V_r|\le C(\e+|\chi_b|)+C\frac{r^{n-1}}{\e\mu}(\sup_{r\ge 1}|V(r)|+\sup_{r\ge 1}|\Theta(r)|),\qquad r\ge 1,\\
&{\rm (iii)}\quad |\alpha-\alpha^E|\le C\e\mu^{\al}+C\e(\sup_{r \ge 1+\mu^{\al}}|V(r)|+\sup_{r \ge 1+\mu^{\al}}|\Theta(r)|),\\
&{\rm (iv)}\quad |\Theta_r(r)|\le C, \qquad r\ge 1.
\end{aligned}
\end{align}
\end{lem}
\noindent
\underline{\it Proof of Lemma \ref{estaV}.} \ 
Rearranging the first equation of \eqref{inf0}, we obtain
\begin{align}
\label{esVr}
\begin{aligned}
|V_r|&=|-\frac{r^{n-1}}{\epsilon\mu}a(r)V-\eta^{E}_r+F|\\
&\le C (|\chi_b|+\e)+C\frac{r^{n-1}}{\e\mu}(\sup_{s\ge 1+\mu^{\alpha}}|V(s)|+\sup_{s\ge 1+\mu^{\alpha}}|\Theta(s)|),
\end{aligned}
\end{align}
which gives (i). The estimate (ii) follows in the same manner
by considering $r\ge 1$. 

Using the definitions of \(\alpha\) and \(\alpha^E\) in (\ref{alpha}) and (\ref{alphaE}), respectively, together with the estimates \eqref{mutonu}, \eqref{4decayE1}, \eqref{decayE2}, and \eqref{esVr}, we obtain the following estimate for \(|\alpha-\alpha^E|\): 
\begin{align}
\begin{aligned}
&|\alpha-\alpha^E|=|\kappa(n-2)\left(H(\eta,\eta_r,\chi)(1)-H^E(\eta^E, \chi^E)(1)\right)|\\
&\le \left|\kappa(n-2)\int_1^{1+\mu^{\alpha}}\frac{\epsilon c_P}{\kappa \tau^{n-1}}\Theta(\tau) + 
\frac{\epsilon^3}{2}\frac{(2v_++\eta+\eta^E)V}{\kappa \tau^{3(n-1)}}\right.\\
&\hspace{2cm}\left. -\epsilon^2\mu \dfrac{(v_++\eta)(\eta^E+V)_\tau}{\kappa \tau^{2(n-1)}}+\frac{2\epsilon^2\nu(n-1)(v_++\eta)^2}{\kappa \tau^{2n-1}}d\tau\right|\\
&+\left|\kappa(n-2)\int_{1+\mu^{\alpha}}^{\infty}\frac{\epsilon c_P}{\kappa \tau^{n-1}}\Theta(\tau) + 
\frac{\epsilon^3}{2}\frac{(2v_++\eta+\eta^E)V}{\kappa \tau^{3(n-1)}}\right.\\
&\hspace{2cm}\left.-\epsilon^2\mu \dfrac{(v_++\eta)(\eta^E+V)_\tau}{\kappa \tau^{2(n-1)}}+\frac{2\epsilon^2\nu(n-1)(v_++\eta)^2}{\kappa \tau^{2n-1}}d\tau\right|\\
&\le C\e\mu^{\al}+C\e\left(\sup_{r \ge 1+\mu^{\al}}|V(r)|+\sup_{r \ge 1+\mu^{\al}}|\Theta(r)|\right).
\end{aligned}
\end{align}
Finally, from the second equation of \eqref{inf0}, together with
(ii) and (iii) in \eqref{nlV}, Theorem~\ref{rmtEI}, and the uniform
boundedness of $\eta$, $\chi$, $\eta^E$, and $\chi^E$, we obtain
\[
|\Theta_r(r)|\le C, \qquad r\ge 1.
\]
This proves (iv).\qquad $\Box$

\noindent
\underline{\it Proof of Theorem \ref{ivli0}.} \ 
Applying Duhamel's principle to the first equation of (\ref{inf0}), we obtain
\begin{align}
\label{intchiinf0}
\begin{aligned}
V(r)&=V(1)e^{-\int_1^r \frac{a(\tau)}{\e\mu}\tau^{n-1}d\tau}
-\int_1^r e^{-\int_{\tau}^{r}\frac{|a(s)|}{\epsilon\mu}s^{n-1}ds}\eta^{E}_r({\tau})\,d\tau
+\int_1^r e^{-\int_{\tau}^{r}\frac{|a(s)|}{\epsilon\mu}s^{n-1}ds}F({\tau})\,d\tau \\[2mm]
&=: I_1+I_2+I_3.
\end{aligned} 
\end{align}
Noting that $r\ge 1+\mu^{\alpha}$, we can estimate $I_1$ as 
\begin{align}
\label{I1bleu}
\begin{aligned}
|I_1| \le |V(1)|e^{-\int_1^r \frac{\delta}{\e\mu}\tau^{n-1} d\tau}
\le |V(1)|e^{-\frac{\delta}{\e \mu}(r-1)}
\le |V(1)|e^{-\frac{\delta}{\e} \frac{1}{\mu^{1-\al}}}
\le C\mu, 
\end{aligned}
\end{align}
where we used the assumption $\al \in (0,1)$. 
Using Theorem \ref{rmtEI} and applying Lemma \ref{hyoka1} (II) with \(l=0\) and \(f=1\), we obtain
\begin{align}
&|I_2|\le  \sup_{r\ge 1}|\eta_r^E(r)| \int_1^r G(r,\tau)\,d\tau \notag\\
 &\le C_0\epsilon\mu  \sup_{r\ge 1}|\eta_r^E(r)| \le  C\epsilon\mu\bigl(|u_b|+|\chi_b|\bigr) \le C\mu .
 \label{I2bleu}
\end{align}
By the definition of \(F[V,\Theta]\), we divide \(I_3\) into two parts \(I_{31}\) and \(I_{32}\) as follows:
\begin{align}
\begin{aligned}
&I_{31}:=-\frac{\e (n-1)}{\mu}\int_1^r e^{-\int_{\tau}^{r}\frac{|a(s)|}{\epsilon\mu}s^{n-1}ds}
\tau^{n-1}\int_{\tau}^{\infty}\frac{V(s)}{s^{2n-1}}ds\,d\tau, \\
&I_{32}:=\int_1^r e^{-\int_{\tau}^{r}\frac{|a(s)|}{\epsilon\mu}s^{n-1}ds}\left\{\frac{\tau^{n-1}}{\epsilon\mu}(\ti{p}(\eta^E+v_+,\chi^E+\theta_++\Theta)-\ti{p}(\eta^E+v_+,\chi^E+\theta_+))\right\}\,d\tau.
\end{aligned}
\end{align}
Using  $a(r)\geq\delta$  and integrating by parts, we estimate $I_{31}$ as follows:
\begin{align}
\label{I3bleu}
\begin{aligned}
|I_{31}|&\le \frac{\e (n-1)}{\mu}\int_1^r e^{-\frac{\delta}{\e\mu n}(r^n-\tau^n)}
\tau^{n-1}\int_{\tau}^{\infty}\frac{|V(s)|}{s^{2n-1}}ds\,d\tau \\
&= \frac{\e^2 (n-1)}{\delta}\int_1^r \left(e^{-\frac{\delta}{\e\mu n}(r^n-\tau^n)}\right)_{\tau}\int_{\tau}^{\infty}\frac{|V(s)|}{s^{2n-1}}ds\ d\tau \\
&\le \frac{\e^2 (n-1)}{\delta}\left(\int_{r}^{\infty}\frac{|V(s)|}{s^{2n-1}}ds+\int_1^r  e^{-\frac{\delta}{\e\mu n}(r^n-\tau^n)} \frac{|V(\tau)|}{\tau^{2n-1}}d\tau \right) \\
&\le \frac{\e^2 (n-1)}{\delta}\left(C \sup_{r\ge 1+\mu^{\alpha}}|V(r)|+C\frac{\e\mu}{\delta}\sup_{r\ge 1}|V(r)|\right)\\
&\le C|u_b|^2(\sup_{r\ge 1+\mu^{\alpha}}|V(r)|+C\mu). 
\end{aligned}
\end{align} 
We obtain the following estimate for $|I_{32}|$:
\begin{equation}
\label{I32}
\begin{aligned}
|I_{32}|
\le{}&
\left|
\int_1^{1+\mu^\alpha}
e^{-\int_\tau^r \frac{|a(s)|}{\epsilon\mu}s^{n-1}\,ds}
\frac{\tau^{n-1}}{\epsilon\mu}
\right.\\
&\qquad\left.
\times
\Bigl(
\widetilde p(\eta^E+v_+,\chi^E+\theta_++\Theta)
-\widetilde p(\eta^E+v_+,\chi^E+\theta_+)
\Bigr)\,d\tau
\right|\\
&+
\left|
\int_{1+\mu^\alpha}^{r}
e^{-\int_\tau^r \frac{|a(s)|}{\epsilon\mu}s^{n-1}\,ds}
\frac{\tau^{n-1}}{\epsilon\mu}
\right.\\
&\qquad\left.
\times
\Bigl(
\widetilde p(\eta^E+v_+,\chi^E+\theta_++\Theta)
-\widetilde p(\eta^E+v_+,\chi^E+\theta_+)
\Bigr)\,d\tau
\right|\\
\le{}&
C\sup_{1\le r\le1+\mu^\alpha}|\Theta(r)|
+
C\sup_{r\ge1+\mu^\alpha}|\Theta(r)|\\
\le{}&
C\mu^\alpha
+
C\sup_{r\ge1+\mu^\alpha}|\Theta(r)|,
\end{aligned}
\end{equation}
where we applied Lemma~\ref{hyoka1}~(II) with
$l=-n+1$ and $f(r)=r^{n-1}$ to both terms. 
Moreover, since $\Theta(1)=0$, it follows from
Lemma~\ref{estaV}~(iv) that
\[
\sup_{1\le r\le 1+\mu^\alpha}|\Theta(r)|
\le C\mu^\alpha.
\]

Combining the estimates (\ref{intchiinf0})--(\ref{I32}), we obtain
\begin{align}
\label{estVl}
\sup_{r\ge 1+\mu^{\alpha}}|V(r)| \le C\mu^{\al}+C\sup_{r\ge 1+\mu^{\al}}|\Theta(r)|.
\end{align}

\bigskip
Next we proceed to the estimate of $|\Theta(r)|$ on $r\in [1+\mu^{\alpha},\infty)$. 
Integrating the second equation of (\ref{inf0}) over \([r,\infty)\), we obtain
\begin{align}
\label{Theta}
&\Theta(r)=-\int_r^{\infty}\left\{\frac{\e c_P}{\kappa s^{n-1}}\Theta(s)+N[V,\Theta](s) \right\} ds,\quad r > 1.
\end{align}
We can estimate the first term on the right-hand side of (\ref{Theta}) as 
\begin{align}
\label{est1Theta}
|\int_r^{\infty}\frac{\e c_P}{\kappa s^{n-1}}\Theta(s)ds|\le C\e\sup_{r\ge 1+\mu^{\alpha}}|\Theta(r)|.
\end{align}
Next, we estimate the second term on the right-hand side of (\ref{Theta}). 
We divide the second term on the right-hand side of (\ref{Theta}) into four terms \(I_1,\ldots,I_4\), i.e.,
\begin{align}
\begin{aligned}
&\int_r^{\infty}N[V,\Theta](s)ds\\
&=\int_r^{\infty}-\frac{\al -\al^E}{\kappa s^{n-1}}+\frac{\e^3}{2\kappa s^{3(n-1)}}(2v_++V+2\eta^E)V\\
&\qquad -\e^2\mu\frac{(v_++V+\eta^E)(V+\eta^E)_s}{\kappa s^{2(n-1)}}+
2\e^2\nu(n-1)\frac{(v_++V+\eta^E)^2}{\kappa s^{2n-1}}ds\\
&=:I_1+I_2+I_3+I_4.
\end{aligned}
\end{align}
Using estimate (iii) in (\ref{nlV}), we can estimate \(I_1\) and \(I_2\) as follows:
\begin{align}
\label{I1I2}
\begin{aligned}
&|I_1|\le |\al -\al^E|\int_r^{\infty}\frac{1}{\kappa s^{n-1}}ds\le C\e\mu^{\al}+C\e\left(\sup_{r \ge 1+\mu^{\al}}|V(r)|+\sup_{r \ge 1+\mu^{\al}}|\Theta(r)|\right),\\
&|I_2|=|\int_r^{\infty}\frac{\e^3}{2\kappa s^{3(n-1)}}(2v_++V+2\eta^E)Vds|
\le C\e\sup_{r\ge 1+\mu^{\al}}|V(r)|.\end{aligned}
\end{align}
Applying (\ref{mutonu}), Theorem \ref{rmtEI}, and Lemma \ref{estaV}(i) to estimate \(I_3\) and \(I_4\), we obtain
\begin{align}
\label{I3I4}
\begin{aligned}
|I_3|&\le C\e\mu+C\e(\sup_{r\ge 1+\mu^{\alpha}}|V(r)|+\sup_{r\ge 1+\mu^{\alpha}}|\Theta(r)|),\\
|I_4|& \le C\e\mu.
\end{aligned}
\end{align}
Combining (\ref{Theta})-(\ref{I3I4}), we obtain that 
\begin{align}
\label{estThetal}
\begin{aligned}
\sup_{r\ge 1+\mu^{\alpha}}|\Theta(r)| \le C\mu^{\al}+C\e(\sup_{r\ge 1+\mu^{\alpha}}|V(r)|+\sup_{r\ge 1+\mu^{\alpha}}|\Theta(r)|). 
\end{aligned}
\end{align}
Combining \eqref{estVl} and \eqref{estThetal}, and taking
$\epsilon>0$ sufficiently small, we obtain \eqref{invsidlimitinf1}.
 \qquad $\Box$



\medskip

\noindent
\underline{\bf Zero-viscosity limit on $I_{b}$}

\smallskip
We define a new variable \(V_b(r)\) as the difference between the solution \(\eta\) of the reformulated problem (\ref{reforme}) and the sum of \(\eta^E\) and the boundary-layer correction, where \(\eta^E\) and \(\eta^B\) are given by (\ref{4reformE}) and (\ref{Narrange3}), respectively, that is, $V_b(r):=\eta(r)-(\eta^{E}(r)+\eta^B(\frac{r-1}{\mu})-\eta^E(1))$. 
To derive the equations for \(V_b(r)\) and \(\Theta(r)\) on \([1,1+\mu^\alpha]\), we subtract (\ref{4reformE}) and (\ref{reformBr}) from (\ref{reforme}) and obtain
\begin{align}
\label{inf0bl}
 \left\{
\begin{aligned}
&(V_b)_r+\frac{r^{n-1}}{\e\mu}a(r)V_b=-\eta^{E}_r+F_1+F_2+F_3,\quad r > 1,\\[2mm]
&\Theta_r-\frac{\e c_p}{\kappa r^{n-1}}\Theta=N[V_b,\Theta](r),\quad r > 1,\\[2mm]
&\displaystyle{\lim_{r\to \infty}}V_b(r) =0,\ V_b(1) =0,\ \displaystyle{\lim_{r\to \infty}}\Theta(r)=\Theta(1) =0,
\end{aligned}
 \right.
\end{align}
where 
\begin{align*}
\begin{aligned}
&a := -\frac{(\ti p(v_++\eta^{E}+\bar \eta+V_b,\Theta+\chi^E+\theta_+)-\ti p(v_++\eta^{E}+\bar \eta,\Theta+\chi^E+\theta_+))}{V_b}-\frac{\e^2}{r^{2(n-1)}},\\
& F_1:=\frac{r^{n-1}}{\e\mu}(\ti p(v_++\eta^{E}+\bar \eta,\Theta+\chi^E+\theta_+)-\ti p(v_++\eta^{E},\Theta+\chi^E+\theta_+))\\
&\qquad -\frac{1}{\e\mu}\left(\ti p(v_++\bar \eta+\eta^{E}(1),\theta_b)-\ti p(v_++\eta^{E}(1),\theta_b)\right)
-\frac{\e}{\mu}(1-\frac{1}{r^{n-1}})\bar\eta,\\
&F_2:=-\frac{\e (n-1)r^{n-1}}{\mu}\int_r^{\infty}\frac{(\eta-\eta^{E})(s)}{s^{2n-1}}ds,\\
&F_3:=\frac{r^{n-1}}{\e\mu}(\ti p(v_++\eta^{E},\theta_++\chi^E+\Theta)-\ti p(v_++\eta^{E},\theta_++\chi^E)),\\
&N[V_b,\Theta](r):=-\frac{\al -\al^E}{\kappa r^{n-1}}+\frac{\e^3}{2\kappa r^{3(n-1)}}(2v_++V_b+\bar{\eta}+2\eta^E)(V_b+\bar{\eta})\\
&\hspace{2cm} -\e^2\mu\frac{(v_++V_b+\bar{\eta}+\eta^E)(V_b+\bar{\eta}+\eta^E)_r}{\kappa r^{2(n-1)}}+
2\e^2\nu(n-1)\frac{(v_++V_b+\bar{\eta}+\eta^E)^2}{\kappa r^{2n-1}}.
\end{aligned}
\end{align*}
By the same arguments as in the previous sections, 
it is easy to see that  
for suitably small \(|u_b|, |v_b-v_+|\), and \(|\theta_b-\theta_+|\), there exists a positive constant \(\delta\), independent of \(\mu\), such that
\begin{equation}
\label{|a|inflow}
a(r) \ge \delta, \qquad r\ge 1.
\end{equation}
We prove the following proposition for the problem (\ref{inf0bl}). 
\begin{pro} 
\label{ivli}Let $n \ge 3$ and $\al \in (0,1)$. Then, for any $v_+>0$ and $\theta_+>0$, there exist positive constants $\epsilon_0$ 
and $C$ which are independent of $\mu\in (0,1]$
such that if $|u_b|+|v_b-v_+|+|\theta_b-\theta_+| \le \epsilon_0$, 
the solution \((V_b,\Theta)\) of (\ref{inf0bl}) satisfies 
\begin{align}
\label{invsidlimitinf2}
\begin{aligned}
&\sup_{1\le r \le 1+\mu^{\al}}|V_b(r)|+\sup_{1\le r \le 1+\mu^{\al}}|\Theta(r)| \le C\mu^{\alpha}.
\end{aligned}
\end{align}
\end{pro}
\noindent
{\it Proof of Proposition \ref{ivli}.}\quad 
Applying Duhamel's principle to the first equation of (\ref{inf0bl}), we obtain
\begin{align}
\label{chiduabl0}
V_b(r)=\int_1^r e^{-\int_{\tau}^r \frac{a(s)}{\e\mu}s^{n-1}ds}
(-\eta^{E}_{r}+F_1+F_2+F_3)(\tau)\,d\tau.
\end{align}
By the same calculation as in (\ref{I2bleu}), 
we can estimate the first term on the right-hand side of (\ref{chiduabl0}) as 
\begin{align}
\label{4first}
\begin{aligned}
&|\int_1^r e^{-\int_{\tau}^r \frac{a(s)}{\e\mu}s^{n-1}ds}\eta^{E}_r({\tau})\,d\tau|
\le \frac{\e\mu}{\delta}\sup_{r\ge 1}|\eta^{E}_r(r)| \le C\mu.
\\
\end{aligned}
\end{align}
Noting that 
\begin{align}
\label{esF1}
\begin{aligned}
|F_1|&=|\frac{r^{n-1}}{\e\mu}(\ti p(v_++\eta^{E}+\bar \eta,\theta_++\chi^E+\Theta)-\ti p(v_++\eta^{E},\theta_++\chi^E+\Theta))\\
&\qquad -\frac{1}{\e\mu}\left(\ti p(v_++\bar \eta+\eta^{E}(1),\theta_b)-\ti p(v_++\eta^{E}(1),\theta_b)\right)
-\frac{\e}{\mu}(1-\frac{1}{r^{n-1}})\bar\eta |\\
&\le |(\frac{r^{n-1}}{\e\mu}-\frac{1}{\e\mu})(\ti p(v_++\bar \eta+\eta^{E},\theta_++\chi^E+\Theta)-\ti p(v_++\eta^{E},\theta_++\chi^E+\Theta))|\\
&+|\frac{1}{\e\mu}\left(\ti p(v_++\bar \eta+\eta^{E},\theta_++\chi^E+\Theta)-\ti p(v_++\bar \eta+\eta^{E}(1),\theta_++\chi^E+\Theta)\right)|\\
&\quad+ |\frac{1}{\e\mu}\left(-\ti p(v_++\eta^{E},\theta_++\chi^E+\Theta)+\ti p(v_++\eta^{E}(1),\theta_++\chi^E+\Theta)\right)|\\
&\quad + |\frac{1}{\e\mu}\left(\ti p(v_++\bar{\eta}+\eta^{E}(1),\theta_++\chi^E+\Theta)-\ti p(v_++\bar{\eta}+\eta^{E}(1),\theta_b)\right)|\\
&\quad + |\frac{1}{\e\mu}\left(-\ti p(v_++\eta^{E}(1),\theta_++\chi^E+\Theta)+\ti p(v_++\eta^{E}(1),\theta_b)\right)|\\
&\quad +|\frac{\e}{\mu r^{n-1}}(r^{n-1}-1)\bar\eta|\\
&\le \frac{C}{\mu\e}|r-1|,
\end{aligned}
\end{align}
the second term on the right-hand side of (\ref{chiduabl0}) can be estimated as \begin{align}
\label{blF1}
\begin{aligned}
&|\int_1^r e^{-\int_{\tau}^r \frac{a(s)}{\e\mu}s^{n-1}ds}F_1(\tau)\,d\tau| 
\le \frac{C}{\mu\e}\int_1^r e^{-\frac{\delta}{\e\mu}(r-\tau)}|\tau-1|\,d\tau 
\le C\mu^{\alpha},
\end{aligned}
\end{align}
where we used the fact that $r\in[1,1+\mu^\alpha]$. Using \eqref{|a|inflow} and integrating by parts, we estimate the contribution of $F_2$ in \eqref{chiduabl0} as follows:
\begin{align}
\label{blF2-1}
\begin{aligned}
&|\int_1^r e^{-\int_{\tau}^r \frac{a(s)}{\e\mu}s^{n-1}ds}F_2(\tau)\,d\tau| \\
&\le \frac{\e (n-1)}{\mu}\int_1^r e^{-\frac{\delta (r^n-\tau^n)}{\e\mu n}}\tau^{n-1}
\int_{\tau}^{\infty}\frac{|(\eta-\eta^{E})(s)|}{s^{2n-1}}ds \, d\tau \\
&\le \frac{\e (n-1)}{\mu}\int_1^r \frac{\e\mu}{\delta}
(e^{-\frac{\delta (r^n-\tau^n)}{\e\mu n}})_{\tau}
\int_{\tau}^{\infty}\frac{|(\eta-\eta^{E})(s)|}{s^{2n-1}}ds \, d\tau \\
&\le \frac{\e^2 (n-1)}{\delta}\left\{\int_{r}^{\infty}\frac{|(\eta-\eta^{E})(s)|}{s^{2n-1}}ds
+\int_1^r  e^{-\frac{\delta (r^n-\tau^n)}{\e\mu n}}
\frac{|(\eta-\eta^{E})(\tau)|}{\tau^{2n-1}} d\tau \right\} \\
&\le \frac{\e^2 (n-1)}{\delta}\left\{\int_{r}^{\infty}\frac{|(\eta-\eta^{E})(s)|}{s^{2n-1}}ds
+\frac{\e\mu}{\delta} \sup_{r\ge 1}|(\eta-\eta^{E})(r)|\right\}\\
&\le C\frac{\e^2(n-1)}{\delta}\left(\int_{r}^{1+\mu^{\al}}\frac{|(\eta-\eta^{E})(s)|}{s^{2n-1}}ds+\int_{1+\mu^{\al}}^{\infty}\frac{|(\eta-\eta^{E})(s)|}{s^{2n-1}}ds\right) + C\mu\\
&\le C\mu^{\al}+C\e\sup_{r\ge 1+\mu^{\al}}|(\eta-\eta^E)(r)|=C\mu^{\al}+C\e\sup_{r\ge 1+\mu^{\al}}|V(r)|.
\end{aligned}
\end{align}
Using \eqref{|a|inflow} and the explicit form of the pressure
$\widetilde p(v,\theta)=R\theta/v$, we estimate the contribution of
$F_3$ in \eqref{chiduabl0} as follows:
\begin{align}
\label{blF2-3}
\begin{aligned}
&|\int_1^r e^{-\int_{\tau}^r \frac{a(s)}{\e\mu}s^{n-1}ds}F_3({\tau})\,d\tau|\\
&=|\int_1^r e^{-\int_{\tau}^r \frac{a(s)}{\e\mu}s^{n-1}ds}\left(\frac{\tau^{n-1}}{\e\mu}(\ti p(v_++\eta^{E},\theta_++\chi^E+\Theta)-\ti p(v_++\eta^{E},\theta_++\chi^E))\right)({\tau})\,d\tau|\\
&\le C \sup_{1\le r\le 1+\mu^{\al}}|\Theta(r)|.
\end{aligned}
\end{align}
Substituting \eqref{4first}, \eqref{blF1}, \eqref{blF2-1}, and
\eqref{blF2-3} into \eqref{chiduabl0}, we obtain
\[
\sup_{1\le r\le 1+\mu^\alpha}|V_b(r)|
\le C\mu^\alpha
 +C\sup_{1\le r\le 1+\mu^\alpha}|\Theta(r)|
 +C\epsilon\sup_{r\ge 1+\mu^\alpha}|V(r)|.
\]
Hence, taking \(\epsilon>0\) sufficiently small, and using \eqref{invsidlimitinf1}, we obtain
\begin{equation}
\label{estVb1}
 \sup_{1\le r\le 1+\mu^\alpha}|V_b(r)|
 \le C\mu^\alpha
 +C\sup_{1\le r\le 1+\mu^\alpha}|\Theta(r)|.
\end{equation}

Next we proceed to the estimate of $|\Theta(r)|$ on $r\in [1, 1+\mu^{\alpha}]$. For this purpose, we establish the following lemma. 
\begin{lem}
\label{1}For $V_b$, $\Theta$ and $\alpha-\alpha^E$ appearing in (\ref{inf0bl}), we have the following estimates.
\begin{align}
\label{estVral}
\begin{aligned}
&{\rm (i)}\quad |(V_b)_r|\le C(\e+|\chi_b|)+C\frac{r^{n-1}}{\e\mu}(\sup_{1\le r\le 1+\mu^{\alpha}}|V_b(r)|+\sup_{1\le r\le 1+\mu^{\alpha}}|\Theta(r)|)+C\frac{\mu^{\al-1}}{\e},\\
&\hspace{9cm}\quad 1\le r\le 1+\mu^{\al},\\
&{\rm (ii)}\quad |\alpha-\alpha^E| \le   C\e\mu^{\al}+C\e\sup_{1\le r \le 1+\mu^{\al}}|\Theta(r)|.
\end{aligned}
\end{align}
\end{lem}

\noindent 
Proof of Lemma \ref{1}.

We first prove (i). 
For $r\in I_b=[1,1+\mu^\alpha]$, we have
\[
\eta-\eta^E=V_b+\bar{\eta}
\quad\text{on } I_b,
\qquad
\eta-\eta^E=V
\quad\text{on } I_\infty.
\]
Hence, from the definition of $F_2$ in \eqref{inf0bl}, we have
\begin{align*}
 |F_2(r)|
 &\le
 C\frac{\epsilon r^{n-1}}{\mu}
 \left\{
 \int_r^{1+\mu^\alpha}
 \frac{|V_b(s)|+|\bar{\eta}(s)|}{s^{2n-1}}\,ds
 +
 \int_{1+\mu^\alpha}^{\infty}
 \frac{|V(s)|}{s^{2n-1}}\,ds
 \right\}                                                   \\
 &\le
 C\frac{\epsilon r^{n-1}}{\mu}\mu^\alpha
 \sup_{1\le s\le1+\mu^\alpha}|V_b(s)|
 +
 C\frac{\epsilon r^{n-1}}{\mu}
 \int_r^{1+\mu^\alpha}|\bar{\eta}(s)|\,ds                   \\
 &\qquad
 +
 C\frac{\epsilon r^{n-1}}{\mu}
 \sup_{s\ge1+\mu^\alpha}|V(s)|
 \int_{1+\mu^\alpha}^{\infty}\frac{1}{s^{2n-1}}\,ds .
\end{align*}
Using \eqref{4decaybl} and \eqref{invsidlimitinf1}, we obtain
\[
 \int_r^{1+\mu^\alpha}|\bar{\eta}(s)|\,ds
 \le C\epsilon\mu,
 \quad
 \sup_{s\ge1+\mu^\alpha}|V(s)|
 \le C\mu^\alpha .
\]
Therefore,
\begin{equation}
\label{estF2Ib}
 |F_2(r)|
 \le
 C\frac{r^{n-1}}{\epsilon\mu}
 \sup_{1\le s\le1+\mu^\alpha}|V_b(s)|
 +C\epsilon
 +C\frac{\mu^{\alpha-1}}{\epsilon},
 \qquad
 1\le r\le1+\mu^\alpha .
\end{equation}

Now, by the first equation of \eqref{inf0bl}, 
\eqref{4decayE1}, the boundedness of $a(r)$, 
\eqref{esF1}, \eqref{estF2Ib}, and the estimate
\[
 |F_3(r)|
 \le
 C\frac{r^{n-1}}{\epsilon\mu}
 \sup_{1\le s\le1+\mu^\alpha}|\Theta(s)|,
\]
we have
\begin{align}
 |(V_b)_r|
 &=
 \left|
 -\eta_r^E
 -\frac{r^{n-1}}{\epsilon\mu}a(r)V_b
 +F_1+F_2+F_3
 \right|                                                     \notag\\
 &\le
 C(\epsilon+|\chi_b|)
 +C\frac{r^{n-1}}{\epsilon\mu}
   \sup_{1\le s\le1+\mu^\alpha}|V_b(s)|
 +C\frac{|r-1|}{\mu\epsilon}                                
 +C\epsilon
 +C\frac{\mu^{\alpha-1}}{\epsilon}
 +C\frac{r^{n-1}}{\epsilon\mu}
   \sup_{1\le s\le1+\mu^\alpha}|\Theta(s)|                  \notag\\
 &\le
 C(\epsilon+|\chi_b|)
 +C\frac{\mu^{\alpha-1}}{\epsilon}
 +C\frac{r^{n-1}}{\epsilon\mu}
 \left(
   \sup_{1\le s\le1+\mu^\alpha}|V_b(s)|
   +
   \sup_{1\le s\le1+\mu^\alpha}|\Theta(s)|
 \right),
 \qquad 1\le r\le1+\mu^\alpha ,
\label{V_br}
\end{align}
where we used $|r-1|\le\mu^\alpha$ in the last inequality.
This proves (i).
On the other hand, from the definitions of $\alpha$ in \eqref{alpha}
and $\alpha^E$ in \eqref{alphaE}, together with
\eqref{mutonu}, \eqref{4decayE1}, \eqref{decayE2},
\eqref{4decaybl}, \eqref{decayetar},
\eqref{invsidlimitinf1}, \eqref{estVb1}, and   \eqref{V_br},
we estimate $|\alpha-\alpha^E|$ as follows:
\begin{align}
|\alpha-\alpha^E|
\le{}&
\kappa(n-2)
\left|
\int_1^{1+\mu^\alpha}
\left\{
 \frac{\epsilon c_P}{\kappa\tau^{n-1}}\Theta(\tau)
 +\frac{\epsilon^3}{2}
  \frac{(2v_+ +\eta+\eta^E)(V_b+\bar\eta)}
       {\kappa\tau^{3(n-1)}}
\right.\right.
\nonumber\\
&\hspace{32mm}\left.\left.
 -\epsilon^2\mu
  \frac{(v_++\eta)(\eta^E+V_b+\bar\eta)_\tau}
       {\kappa\tau^{2(n-1)}}
 +2\epsilon^2\nu(n-1)
  \frac{(v_++\eta)^2}
       {\kappa\tau^{2n-1}}
\right\}
\,d\tau
\right|
\nonumber\\
&+
\kappa(n-2)
\left|
\int_{1+\mu^\alpha}^{\infty}
\left\{
 \frac{\epsilon c_P}{\kappa\tau^{n-1}}\Theta(\tau)
 +\frac{\epsilon^3}{2}
  \frac{(2v_+ +\eta+\eta^E)V}
       {\kappa\tau^{3(n-1)}}
\right.\right.
\nonumber\\
&\hspace{32mm}\left.\left.
 -\epsilon^2\mu
  \frac{(v_++\eta)(\eta^E+V)_\tau}
       {\kappa\tau^{2(n-1)}}
 +2\epsilon^2\nu(n-1)
  \frac{(v_++\eta)^2}
       {\kappa\tau^{2n-1}}
\right\}
\,d\tau
\right|
\nonumber\\
\le{}&
C\epsilon\mu^\alpha
+C\epsilon
\left(
 \sup_{1\le r\le1+\mu^\alpha}|V_b(r)|
 +\sup_{1\le r\le1+\mu^\alpha}|\Theta(r)|
\right)+
C\epsilon
\left(
 \sup_{r\ge1+\mu^\alpha}|V(r)|
 +\sup_{r\ge1+\mu^\alpha}|\Theta(r)|
\right)
\nonumber\\
\le{}&
C\epsilon\mu^\alpha
+C\epsilon
\sup_{1\le r\le1+\mu^\alpha}|\Theta(r)|. \qquad \Box
\label{alphaestIb}
\end{align}


\noindent
\underline{\it We continue the proof of Proposition \ref{ivli}.}\\
\noindent
Integrating the second equation of \eqref{inf0bl} over $[1,r]$
and using $\Theta(1)=0$, we obtain
\begin{align}
\label{bTheta_b}
&\Theta(r)=\int_1^{r}\left\{ \frac{\e c_P}{\kappa s^{n-1}}\Theta(s)+N[V_b,\Theta](s)\right\} ds,\quad 1\le r \le 1+\mu^{\alpha}.
\end{align}
The first term on the right-hand side of (\ref{bTheta_b}) is estimated as
\begin{align}
\label{best1Theta_b}
|\int_1^{r}\frac{\e c_P}{\kappa s^{n-1}}\Theta(s)ds|\le C\e\sup_{1\le r \le 1+\mu^{\alpha}}|\Theta(r)|.
\end{align}
We decompose the second term on the right-hand side of
\eqref{bTheta_b} into four terms as
\begin{align}
\int_1^r N[V_b,\Theta](s)\,ds
={}&
-\int_1^r \frac{\alpha-\alpha^E}{\kappa s^{n-1}}\,ds
\nonumber\\
&+
\int_1^r
\frac{\epsilon^3}
     {2\kappa s^{3(n-1)}}
(2v_+ + V_b+\bar\eta+2\eta^E)(V_b+\bar\eta)\,ds
\nonumber\\
&-
\int_1^r
\epsilon^2\mu
\frac{(v_+ +V_b+\bar\eta+\eta^E)
      (V_b+\bar\eta+\eta^E)_s}
     {\kappa s^{2(n-1)}}\,ds
\nonumber\\
&+
\int_1^r
2\epsilon^2\nu(n-1)
\frac{(v_+ +V_b+\bar\eta+\eta^E)^2}
     {\kappa s^{2n-1}}\,ds
\nonumber\\
={}& I_1+I_2+I_3+I_4 .
\label{bNdecomp}
\end{align}
Using Lemma~\ref{1} (ii), the estimate \eqref{4decaybl},
and the uniform boundedness of the solutions, we estimate
$I_1$ and $I_2$ as
\begin{align}
|I_1|
&=
\left|
\int_1^r
\frac{\alpha-\alpha^E}{\kappa s^{n-1}}\,ds
\right|
\le
C\epsilon\mu^\alpha
+
C\epsilon
\sup_{1\le s\le1+\mu^\alpha}|\Theta(s)|,
\nonumber\\
|I_2|
&=
\left|
\int_1^r
\frac{\epsilon^3}{2\kappa s^{3(n-1)}}
(2v_++V_b+\bar\eta+2\eta^E)(V_b+\bar\eta)\,ds
\right|
\nonumber\\
&\le
C\epsilon
\sup_{1\le s\le1+\mu^\alpha}|V_b(s)|+C\epsilon\mu^\alpha
.
\label{bI1I2}
\end{align}
Using \eqref{estVral}, \eqref{decayE2}, \eqref{decayetar},
and \eqref{estVb1}, together with the uniform boundedness of the
solutions, we estimate $I_3$ as
\begin{align}
|I_3|
&=
\left|
\int_1^r
\epsilon^2\mu
\frac{
(v_+ + V_b+\bar\eta+\eta^E)
(V_b+\bar\eta+\eta^E)_s
}{
\kappa s^{2(n-1)}
}\,ds
\right|
\nonumber\\
&\le
C\epsilon\mu^\alpha
+
C\epsilon
\left(
\sup_{1\le s\le1+\mu^\alpha}|V_b(s)|
+
\sup_{1\le s\le1+\mu^\alpha}|\Theta(s)|
\right)
\nonumber\\
&\le
C\epsilon\mu^\alpha
+
C\epsilon
\sup_{1\le s\le1+\mu^\alpha}|\Theta(s)|.
\nonumber\\
\intertext{Moreover, by \eqref{mutonu} and $r-1\le\mu^\alpha$, we obtain}
|I_4|
&=
\left|
\int_1^r
2\epsilon^2\nu(n-1)
\frac{
(v_+ + V_b+\bar\eta+\eta^E)^2
}{
\kappa s^{2n-1}
}\,ds
\right|
\nonumber\\
&\le
C\epsilon^2\nu\mu^\alpha
\le
C\epsilon\mu .
\label{bI3I4}
\end{align}
Combining (\ref{bTheta_b}), (\ref{best1Theta_b}), (\ref{bI1I2}) and (\ref{bI3I4}), we obtain that 
\begin{align*}
\begin{aligned}
\sup_{1\le r\le 1+\mu^{\alpha}}|\Theta(r)|\le C\mu^{\al}+C\e(\sup_{1\le r\le 1+\mu^{\al}}|\Theta(r)|+\sup_{1\le r\le 1+\mu^{\al}}|V_b(r)|).
\end{aligned}
\end{align*}

Taking $\epsilon>0$ sufficiently small, we obtain
\[
\sup_{1\le r\le1+\mu^\alpha}|\Theta(r)|
\le
C\mu^\alpha
+
C\epsilon
\sup_{1\le r\le1+\mu^\alpha}|V_b(r)|.
\]
Using \eqref{estVb1} and taking $\epsilon>0$ sufficiently small once again,
we obtain
\begin{equation}
\label{estThetalb}
\sup_{1\le r\le1+\mu^\alpha}|\Theta(r)|
\le C\mu^\alpha.
\end{equation}
Finally, combining \eqref{estVb1} and \eqref{estThetalb}, we obtain
\[
\sup_{1\le r\le1+\mu^\alpha}|V_b(r)|
+
\sup_{1\le r\le1+\mu^\alpha}|\Theta(r)|
\le C\mu^\alpha,
\]
which proves \eqref{invsidlimitinf2}. \qquad \qquad $\Box$

\medskip

\subsection{Zero-viscosity limit for the inflow problem on
$I=[1,\infty)$}

We now combine the estimates on $I_b$ and $I_\infty$.
From \eqref{invsidlimitinf1}, we have
\[
 \sup_{r\ge1+\mu^\alpha}|V(r)|
 +
 \sup_{r\ge1+\mu^\alpha}|\Theta(r)|
 \le C\mu^\alpha,
\]
while \eqref{invsidlimitinf2} gives
\[
 \sup_{1\le r\le1+\mu^\alpha}|V_b(r)|
 +
 \sup_{1\le r\le1+\mu^\alpha}|\Theta(r)|
 \le C\mu^\alpha.
\]
\noindent
On the other hand, we have
\[
 V_b=V-\bar\eta
 \qquad\text{on } I_\infty.
\]
Since exponential decay is faster than any algebraic rate,
it follows from \eqref{4decaybl} that, for $r\ge1+\mu^\alpha$,
\[
|\bar\eta(r)|
\le
C\exp\left\{-\frac{\nu_0\mu^{\alpha-1}}{\epsilon}\right\}
\le C\mu^\alpha.
\]
Combining the above estimate with \eqref{invsidlimitinf1}
and \eqref{invsidlimitinf2}, we obtain
\begin{equation}
\label{globalVbTheta}
 \sup_{r\ge1}|V_b(r)|
 +
 \sup_{r\ge1}|\Theta(r)|
 \le C\mu^\alpha.
\end{equation}

Recalling the definitions
\[
 V_b=\eta-\eta^E-\bar\eta,
 \qquad
 \Theta=\chi-\chi^E,
\]
we therefore obtain
\begin{equation}
\label{limitonr}
 \sup_{r\ge1}
 |\eta(r)-\eta^E(r)-\bar\eta(r)|
 +
 \sup_{r\ge1}
 |\chi(r)-\chi^E(r)|
 \le C\mu^\alpha. 
\end{equation}

\medskip

Recalling the definition
\[
 \bar\eta(r)
 =
 \eta^B\left(\frac{r-1}{\mu}\right)-\eta^E(1),
\]
the first estimate in \eqref{limitonr} can be rewritten as
\begin{equation}
\label{limitetaB}
 \sup_{r\ge1}
 \left|
 \eta(r)-\eta^E(r)
 -\eta^B\left(\frac{r-1}{\mu}\right)
 +\eta^E(1)
 \right|
 \le C\mu^\alpha .
\end{equation}

We now return to the original variables.
Set
\[
 y:=\frac{r-1}{\mu},
 \qquad
 R(r):=\eta(r)-\eta^E(r)-\bar\eta(r).
\]
Recall that
\[
 \rho(r)=\frac{1}{v_++\eta(r)},\qquad
 \rho^E(r)=\frac{1}{v_++\eta^E(r)},\qquad
 \rho^B(y)=\frac{1}{v_++\eta^B(y)}.
\]

Since the quantities under consideration remain uniformly in a sufficiently small neighborhood of \(v_+\), all the quantities appearing in the denominators below are uniformly bounded away from zero. Moreover,
\[
 v_+ + \eta(r)
 =
 v_+ + \eta^E(r) + \bar\eta(r) + R(r),
\]
and
\[
 v_+ + \eta^B(y)
 =
 v_+ + \eta^E(1) + \bar\eta(r).
\]
Hence,
\begin{align*}
&
\left|
 \rho(r)-\rho^E(r)-\rho^B(y)+\rho^E(1)
\right|
\\
&\le
\left|
 \frac{1}{v_++\eta(r)}
 -
 \frac{1}{v_++\eta^E(r)+\bar\eta(r)}
\right|
\\
&\quad
+
\left|
 \frac{1}{v_++\eta^E(r)+\bar\eta(r)}
 -
 \frac{1}{v_++\eta^E(r)}
 -
 \frac{1}{v_++\eta^E(1)+\bar\eta(r)}
 +
 \frac{1}{v_++\eta^E(1)}
\right|.
\end{align*}
The first term on the right-hand side is bounded by
\[
 C|R(r)|.
\]
For the second term, putting
\[
 a:=v_++\eta^E(r),\qquad
 c:=v_++\eta^E(1),\qquad
 b:=\bar\eta(r),
\]
we have
\begin{align*}
&
\left|
 \frac{1}{a+b}-\frac{1}{a}
 -\frac{1}{c+b}+\frac{1}{c}
\right|
\\
&=
|b|
\left|
 \frac{1}{c(c+b)}
 -
 \frac{1}{a(a+b)}
\right|
\le
C|a-c|\,|b|.
\end{align*}
Therefore,
\[
\left|
 \rho(r)-\rho^E(r)-\rho^B(y)+\rho^E(1)
\right|
\le
C|R(r)|
+
C|\eta^E(r)-\eta^E(1)|\,|\bar\eta(r)|.
\]
By \eqref{decayE2} and \eqref{4decaybl},
\[
 |\eta^E(r)-\eta^E(1)|\,|\bar\eta(r)|
 \le
 C(r-1)
 \exp\left\{
 -\frac{\nu_0(r-1)}{\epsilon\mu}
 \right\}
 \le C\epsilon\mu
 \le C\mu^\alpha.
\]
Together with \eqref{limitonr}, this yields
\begin{equation}
\label{limitdensity}
 \sup_{r\ge1}
 \left|
 \rho(r)-\rho^E(r)
 -\rho^B\left(\frac{r-1}{\mu}\right)
 +\rho^E(1)
 \right|
 \le C\mu^\alpha .
\end{equation}

Next, since $\epsilon=u_b/v_b$ for the inflow problem, we have
\[
 u(r)
 =
 \frac{\epsilon(v_++\eta(r))}{r^{n-1}},
 \qquad
 u^E(r)
 =
 \frac{\epsilon(v_++\eta^E(r))}{r^{n-1}},
\]
and
\[
 u^B(y)
 =
 \epsilon(v_++\eta^B(y)),
 \qquad
 u^E(1)
 =
 \epsilon(v_++\eta^E(1)).
\]
Therefore,
\begin{align*}
&
u(r)-u^E(r)-u^B(y)+u^E(1)
\\
&\qquad=
\epsilon
\left\{
 \frac{R(r)}{r^{n-1}}
 +
 \bar\eta(r)
 \left(
 \frac{1}{r^{n-1}}-1
 \right)
\right\}.
\end{align*}
Since
\[
 \left|
 \frac{1}{r^{n-1}}-1
 \right|
 \le C(r-1),
 \qquad r\ge1,
\]
it follows from \eqref{4decaybl} that
\[
 |\bar\eta(r)|
 \left|
 \frac{1}{r^{n-1}}-1
 \right|
 \le
 C(r-1)
 \exp\left\{
 -\frac{\nu_0(r-1)}{\epsilon\mu}
 \right\}
 \le C\mu^\alpha .
\]
Thus, by \eqref{limitonr},
\begin{equation}
\label{limitvelocity}
 \sup_{r\ge1}
 \left|
 u(r)-u^E(r)
 -u^B\left(\frac{r-1}{\mu}\right)
 +u^E(1)
 \right|
 \le C\mu^\alpha .
\end{equation}

Finally, since
\[
 \theta(r)=\theta_++\chi(r),
 \qquad
 \theta^E(r)=\theta_++\chi^E(r),
\]
and the boundary layer temperature satisfies
\[
 \theta^B(y)\equiv\theta_b=\theta^E(1),
\]
the second estimate in \eqref{limitonr} directly gives
\begin{equation}
\label{limittemperature}
 \sup_{r\ge1}
 |\theta(r)-\theta^E(r)|
 \le C\mu^\alpha .
\end{equation}

Combining \eqref{limitdensity}, \eqref{limitvelocity}, and
\eqref{limittemperature}, we obtain the estimates in statement
(III) of Theorem~\ref{bl+E}. This completes the proof.

%

\section*{Acknowledgments}

I am deeply grateful to Professor Akitaka Matsumura
of Osaka University for many helpful discussions
and valuable suggestions concerning the calculations
in this paper.
I also sincerely thank Professor Seiro Omata
of Kanazawa University for helpful discussions
and encouragement at the outset of my research
on spherically symmetric problems.


\bigskip

\end{document}